\documentclass[11pt,oneside,reqno]{amsart}
\usepackage[a4paper,hmargin=27mm,vmargin=28mm]{geometry}
\usepackage[T1]{fontenc}
\usepackage{lmodern,amssymb,mathtools,mathrsfs,bm}
\usepackage{microtype,booktabs,graphicx,array,tabularx,enumitem}
\usepackage{aliascnt,cite,caption,subcaption,tikz,float,placeins}
\usetikzlibrary{arrows.meta,positioning,calc}
\usepackage[unicode,colorlinks,linkcolor=black,citecolor=black,urlcolor=blue!45!black,
            bookmarksopen=false,bookmarksdepth=2]{hyperref}
\usepackage[capitalise,nameinlink,noabbrev]{cleveref}
\allowdisplaybreaks[3]
\setlist{nosep,leftmargin=1.7em}
\makeatletter
\renewcommand{\@settitle}{%
  \begin{center}\Large\bfseries
  \baselineskip=20pt\relax
  \@title\par
  \end{center}}
\renewcommand{\@setauthors}{%
  \begin{center}\normalsize\authors\par\end{center}}
\makeatother
 \newtheorem{theorem}{Theorem}[section]
\newaliascnt{lemma}{theorem}\newtheorem{lemma}[lemma]{Lemma}\aliascntresetthe{lemma}
\newaliascnt{proposition}{theorem}\newtheorem{proposition}[proposition]{Proposition}\aliascntresetthe{proposition}
\newaliascnt{corollary}{theorem}\newtheorem{corollary}[corollary]{Corollary}\aliascntresetthe{corollary}
\theoremstyle{definition}
\newtheorem{definition}[theorem]{Definition}
\theoremstyle{remark}

\crefname{lemma}{Lemma}{Lemmas}\crefname{proposition}{Proposition}{Propositions}
\crefname{corollary}{Corollary}{Corollaries}
\newcommand{\R}{\mathbb R}\newcommand{\PP}{\mathbb P}
\newcommand{\HH}{\mathcal H}\newcommand{\im}{\operatorname{im}}
\newcommand{\rank}{\operatorname{rank}}\newcommand{\coker}{\operatorname{coker}}
\newcommand{\Tr}{\operatorname{Tr}}\newcommand{\dist}{\operatorname{dist}}

\newcommand{\Dstar}{\Delta_*}\newcommand{\Dhat}{\widehat\Delta}
\newcommand{\Sspace}{S_3^1}\newcommand{\ee}{\mathcal E_I}
\newcommand{\norm}[1]{\lVert#1\rVert}
\newcommand{\dd}{\delta_V}
\numberwithin{equation}{section}

\hypersetup{
 pdftitle={Vertex interpolation by bivariate C1 cubic splines: counterexamples, geometric stability, and indecomposable defects},
 pdfauthor={Junkai Qiu},
 pdfsubject={Mathematics preprint}}
\title[Vertex interpolation by cubic splines]{Vertex interpolation by bivariate \(C^1\) cubic splines:\\
counterexamples, geometric stability,\\
and indecomposable defects}
\author{Junkai Qiu}
\address{School of Mathematical Sciences, Dalian University of Technology,
Dalian, Liaoning 116024, P.~R.~China}
\email{qjk@mail.dlut.edu.cn}
\date{November 6, 2026}
\begin{document}
\begin{abstract}
We study vertex evaluation on the full space of continuously differentiable
piecewise cubic polynomials on a fixed planar triangulation. A complete dual
description turns failure of interpolation into balanced edge weights with
nonzero vertex loading. It leads to exact counterexamples with 17, 25, 31 and
41 vertices, each with one relation and the classical lower-bound dimension.
The 31- and 41-vertex examples have interior graphs that are trees.
We give a geometric construction criterion for a sector-completion family
and a certified implicit construction outside that family.
With all boundary vertices fixed, the nearby bad configurations form a
four-dimensional smooth locus in the 17-vertex case, two intersecting
components of dimensions six and five in the 25-vertex case, and isolated
transverse points in the two tree examples. A common Schur-complement estimate
explains their different interpolation costs.
We first realize arbitrary defect by decomposable chains of the 41-vertex
example. We then construct obstructions from individually interpolatory
blocks, prove that the capacity of one complete edge is exactly five, and
build a fixed-square family with arbitrarily large strongly indecomposable
defect. The latter excludes two triangle-disjoint connected normal bad
submeshes even when those submeshes may have holes.
\end{abstract}
\maketitle
\thispagestyle{plain}
\tableofcontents
\clearpage
\section{Introduction}
Let \(\Delta\) be a finite conforming straight-line triangulation of a
polygonal disk \(\Omega\). Every triangle has positive area, and distinct
triangles meet, if at all, in a common edge or vertex. We consider
\[
 S_3^1(\Delta)=\{s\in C^1(\Omega):s|_T\in\PP_3,\ T\in\Delta\},
 \qquad E_\Delta s=(s(v))_{v\in\mathcal V(\Delta)}.
\]
Here \(\PP_d\) denotes the bivariate polynomials of total degree at most \(d\).
The space is defined on the original triangles, with no prescribed
boundary values or derivatives. Write \(V,V_B,V_I,N,E_I\) for the numbers of
vertices, boundary vertices, interior vertices, triangles and interior edges.
The vertex defect and relation space are
\[
 \dd(\Delta)=V-\rank E_\Delta,\qquad
 \mathcal O_V(\Delta)=(\im E_\Delta)^\circ.
\]
Thus the question is whether the vertex evaluations are independent, rather
than whether one particular data vector can be interpolated.

Alfeld asked whether \(E_\Delta\) is always surjective for cubic
\(C^1\) splines \cite[Conjecture~11]{Alfeld2000}
\cite[Conjecture~4]{Alfeld2016}. Equivalently, can all vertices be included
in a Lagrange interpolation set for the full spline space?
The dimension lower bound \cite{Schumaker1979}
\begin{equation}\label{eq:lowerbound}
 \dim S_3^1(\Delta)\ge 3V_B+2V_I+\sigma+1
\end{equation}
does not answer this question. Here \(\sigma\) counts singular interior
vertices, whose four incident edges lie on two lines. The lower bound exceeds
\(V\), but a space of dimension larger than \(V\) may still have dependent
vertex evaluations.

The Bernstein--B\'ezier description in
\cite{Alfeld2000,Alfeld2016,LaiSchumaker2007} makes that distinction precise.
After sharing vertex values and gradients, each triangle contributes one
central coefficient and each interior edge contributes one remaining
smoothness equation. Transposing these equations gives scalar edge weights.
The desired counterexample has zero face sums and zero gradient loading,
but nonzero value loading. This criterion involves the whole spline space.
It does not replace a vertex obstruction by a dimension anomaly or by a
dependence among different functionals.

The earlier 41-vertex construction is retained in
\cref{sec:core,sec:ports}. The present account places it in a wider
construction and stability theory. Its basic examples are as follows.
\begin{table}[ht]
\centering
\caption{Four exact counterexamples. The support column counts nonzero
coefficients of the unique vertex relation, not vertices of a smaller mesh.}
\label{tab:four}
\begin{tabular}{@{}rrrrrrrl@{}}
\toprule
\(V\)&\(V_B\)&\(V_I\)&\(N\)&\(E_I\)&\(\dim S_3^1\)&support&interior graph\\
\midrule
17&4&13&28&40&39&17&contains cycles\\
25&6&19&42&60&57&13&contains cycles\\
31&18&13&42&54&81&13&tree, maximum degree 6\\
41&24&17&56&72&107&17&tree, maximum degree 8\\
\bottomrule
\end{tabular}
\end{table}
In every row, \(\sigma=0\), the spline dimension equals
\(3V_B+2V_I+1\), and the evaluation rank is \(V-1\).
The last three meshes have rational coordinates. The first has an exact
implicit definition by rational equations and an isolating box.
In particular, 17 is an upper bound for the smallest counterexample, and 31
is an upper bound within the interior-tree class. Neither statement asserts
global minimality.

\paragraph{Construction and local geometry.}
Section~\ref{sec:design} explains how the balanced-weight equations lead to
the examples. For a complete sector family containing the 25-, 31- and
41-vertex meshes, the obstruction is equivalent to chord ratios, specified
outer lines and one global closure condition. Dropping those geometric
restrictions and retaining the radial weights permits the 17-vertex
construction. This is a reproducible design route; we make no historical
claim about the order in which numerical candidates were originally found.

Fixing a boundary means fixing its labelled planar coordinates.
It imposes no function boundary condition. Under that convention the
17-vertex bad locus is a smooth four-dimensional manifold, the 25-vertex
locus is the union of a smooth six-dimensional component and a singular
five-dimensional component, and the two tree examples are isolated.
Section~\ref{sec:stability} gives exact local statements and a uniform
all-data cost estimate. The singular 25-vertex locus gives paths with
different inverse-distance orders.

\paragraph{From additive to indecomposable defects.}
We first keep the chain construction of Section~\ref{sec:ports}:
it realizes every positive integer as the defect, with an interior tree
and a minimum-angle bound independent of the chain length.
Its independent cores make the defect decomposable.
Sections~\ref{sec:interfaces} and \ref{sec:crossing} then show that large
defect need not arise from disjoint bad pieces. Individually interpolatory
blocks can constrain a shared gradient. One complete edge has an exact
five-dimensional obstruction capacity, attained by a 168-vertex mesh.
Finally, pairwise-crossing transport loops give the following result.

\begin{theorem}[Strongly indecomposable defects are unbounded]
\label{thm:main-unbounded}
For each integer \(n\ge1\), there is a triangulation \(\mathcal D_n\) of
the square \([-1,3]^2\), with only its four corners on the boundary, such that
\[
 V_n=101n^2+247n+12,\qquad
 \dd(\mathcal D_n)=n,\qquad
 \dim S_3^1(\mathcal D_n)=2V_n+5.
\]
There do not exist two connected normal bad submeshes with disjoint
triangle sets. The submeshes may have holes.
\end{theorem}
The proof appears in Section~\ref{sec:crossing}. Normality means that the
triangles at every vertex form a single interval or a full circle;
pieces joined only at a point are excluded. The statement is geometric
and stronger than a claim about disjoint supports of relation vectors.

\paragraph{Related work and evidence.}
Nodal constructions for \(S_d^1\), \(d\ge5\)
\cite{MorganScott1975}, bases for \(S_4^1\)
\cite{AlfeldPiperSchumaker1987}, and cubic schemes on specially constructed
meshes \cite{NurnbergerSchumakerZeilfelder2004}
do not settle interpolation on every fixed unrefined triangulation.
Quadratic geometric obstructions \cite{AllenEtAl2016}
concern a different degree. The instability results of
\cite{PetersSitharam1993} also motivate distinguishing solvability from a
uniformly bounded right inverse. Here the distinction is quantified on
fixed-boundary, constant-dimension families.

The explicit rational examples and local elimination formulas are given
in the text. Computer-assisted parts specify the exact systems, rational
interval tests and matrix minors used. The ancillary certificate files, listed in
\texttt{certificate-index.json}, preserve the underlying certificates,
equations and verification code. The exact 17-variable system and its rational
isolating box are also supplied as \texttt{example17-implicit.json}.
Finite checks certify their stated inputs and local algebra. The infinite
families are justified by the geometric and complete-trace arguments.
 \section{Smoothness conditions and vertex relations}\label{sec:duality}
Following \cite{Alfeld2000,Alfeld2016,LaiSchumaker2007}, we write
\(c_{ijk}\) and \(\widetilde c_{ijk}\) for the cubic B\'ezier coefficients of
triangles \(T=(v_1,v_2,v_3)\) and \(\widetilde T=(v_1,v_2,v_4)\)
sharing the edge \(v_1v_2\).
Let \((b_1,b_2,b_3)\) be the barycentric coordinates of \(v_4\) relative to \(T\).
Across \(v_1v_2\), the smoothness conditions \cite[Theorem~1]{Alfeld2016} are
\begin{equation}\label{eq:B-smoothness}
\begin{aligned}
C^0:\ \widetilde c_{ij0}&=c_{ij0}, &&i+j=3,\\
C^1:\ \widetilde c_{ij1}
 &=b_1c_{i+1,j,0}+b_2c_{i,j+1,0}+b_3c_{ij1}, &&i+j=2.
\end{aligned}
\end{equation}
Taken together, the two rows characterize \(C^1\)
continuity across \(v_1v_2\).

For each vertex \(v\), let \(p_v\in\R^2\) be its position.
Let \(\ee\) be the set of interior edges and put
\([ijk]=\det(p_j-p_i,p_k-p_i)\).
Writing \(A_T=\operatorname{area}(T)\), we have \([ijk]=2A_T\)
for positively oriented \(T=(i,j,k)\).
For \(e=ij\in\ee\), label its incident triangles \(T=(i,j,k)\)
and \(\widetilde T=(i,j,l)\), choose the orientation of \(ij\) so that
\([ijk]>0\), and define
\begin{equation}\label{eq:a}
 \alpha_e=[ijk]>0,\ \beta_e=[jil]>0,\
 a_{ij}=-\frac{[jkl]}{\alpha_e\beta_e},\
 a_{ji}=-\frac{[ilk]}{\alpha_e\beta_e}.
\end{equation}
Reversing the orientation of \(e\) swaps \(i\) with \(j\)
and \(k\) with \(l\) and leaves the coefficient attached to each
endpoint unchanged.

For this edge, the barycentric coordinates in \eqref{eq:B-smoothness} are
\(b_1=[jkl]/\alpha_e\), \(b_2=[ilk]/\alpha_e\), \(b_3=-\beta_e/\alpha_e\).
The coefficients \(c\) and \(\widetilde c\) are indexed by these vertex
orderings, with the same ordering along the common edge.

On \(T=(i,j,k)\), a cubic \(p_T\) is determined by its values
\(f_v=p_T(p_v)\in\R\) and gradients \(g_v=\nabla p_T(p_v)\in\R^2\)
at \(v=i,j,k\), together with the central B\'ezier coefficient
\(q_T=c_{111}\in\R\).
Likewise \(q_{\widetilde T}=\widetilde c_{111}\); these coefficients need not
equal the values of the cubics at the barycenter.
The first-order jet at \(v\) is \(j^1s(v)=(s(v),s_x(v),s_y(v))\).
In particular, the coefficients adjacent to the endpoints of \(ij\) are
\[
 c_{210}=f_i+\tfrac13g_i\cdot(p_j-p_i),\
 c_{120}=f_j+\tfrac13g_j\cdot(p_i-p_j).
\]
Collect the data \((f,g,q)=((f_v)_v,(g_v)_v,(q_T)_T)\)
in \(X=\R^{\mathcal V}\oplus Z\),
\(Z=(\R^2)^{\mathcal V}\oplus\R^\Delta\), where \(\mathcal V=\mathcal V(\Delta)\).
Define \(H:X\to\R^{\ee}\) by \(H(f,g,q)=(H_e(f,g,q))_{e\in\ee}\), with
\begin{equation}\label{eq:He}
 H_e(f,g,q)=\frac{q_T}{\alpha_e}+\frac{q_{\widetilde T}}{\beta_e}
 +a_{ij}\bigl(f_i+\tfrac13g_i\cdot(p_j-p_i)\bigr)
 +a_{ji}\bigl(f_j+\tfrac13g_j\cdot(p_i-p_j)\bigr).
\end{equation}
Here \(H_e\) is the row functional for edge \(e\).
The \(E_I\times(3V+N)\) matrix \(H=[H_f\ H_0]\) is partitioned into
vertex-value columns and gradient- and central-coefficient columns.
The superscript \(\mathsf T\) denotes the transpose.

\begin{lemma}\label{lem:H}
The above parametrization gives \(\Sspace(\Delta)\cong\ker H\) under the
correspondence \(s\leftrightarrow(f,g,q)\), and \(E_\Delta s=f\).
\end{lemma}
\begin{proof}
Vertex jets and \(q_T\) determine each cubic. Equal endpoint values and tangential
derivatives make the two cubic edge traces coincide; equality of the normal
derivatives at the endpoints then gives the endpoint equations of the
\(C^1\) row in \eqref{eq:B-smoothness}. The remaining condition
\(q_{\widetilde T}=([jkl]c_{210}+[ilk]c_{120}-\beta_eq_T)/\alpha_e\)
reduces to \(H_e=0\) in \eqref{eq:He} upon substitution of the endpoint coefficients.
\end{proof}

Assign a scalar \(r_e\) to each unoriented interior edge and set \(r_e=0\)
on boundary edges. For every triangle and every vertex, including boundary
vertices, define
\begin{equation}\label{eq:FGR}
 F_T(r)=\sum_{e\subset T}r_e,\
 G_v(r)=\sum_{w:vw\in\ee}a_{vw}r_{vw}(p_w-p_v),\
 R_v(r)=\sum_{w:vw\in\ee}a_{vw}r_{vw}.
\end{equation}
We omit the argument \(r\) where no confusion arises. In
\(r^{\mathsf T}H(f,g,q)\), the coefficients of \(q_T,f_v,g_v\) are,
respectively, \(F_T/(2A_T)\), \(R_v\), \(G_v/3\).
Consequently
\[
 r\in\ker H_0^{\mathsf T}\iff F_T=0\ \text{for all }T\text{ and }
 G_v=0\ \text{for all }v.
\]
Call such \(r\) a \emph{balanced weight}, and write
\(R(r):=H_f^{\mathsf T}r=(R_v)_v\).
For a subspace \(W\subset U\), its annihilator is
\(W^\circ=\{\lambda\in U^*: \lambda(w)=0\text{ for all }w\in W\}\).

\begin{proposition}\label{prop:duality}
There is a short exact sequence
\begin{equation}\label{eq:exact-dual}
 0\longrightarrow\ker H^{\mathsf T}\longrightarrow\ker H_0^{\mathsf T}
 \xrightarrow{\,R\,}(\im E_\Delta)^\circ\longrightarrow0.
\end{equation}
Consequently \((\coker E_\Delta)^*\cong\ker H_0^{\mathsf T}/\ker H^{\mathsf T}\), and
\begin{equation}\label{eq:dimensions}
 \delta_V=\rank H-\rank H_0,\
 \dim\Sspace=3V_B+2V_I+1+\dim\ker H^{\mathsf T}.
\end{equation}
\end{proposition}
\begin{proof}
A vertex functional \(\lambda\) annihilates \(\im E_\Delta\) exactly when
\((\lambda,0)\in(\ker H)^\circ=\im H^{\mathsf T}\subset X^*\).
Equivalently \(H_f^{\mathsf T}r=\lambda,\ H_0^{\mathsf T}r=0\) for some \(r\).
The kernel of the resulting map \(R\) is \(\ker H^{\mathsf T}\).
Taking dimensions gives the first formula. Since
\(\dim\Sspace=3V+N-E_I+\dim\ker H^{\mathsf T}\),
Euler's formula \(E_I-N=V_I-1\) gives the second formula.
\end{proof}
Thus \(E_\Delta\) is surjective exactly when every balanced weight
\(r\) satisfies \(R(r)=0\).
 \section{Designing an obstruction from balanced weights}\label{sec:design}
\subsection{The seven-triangle sector}
The dual criterion suggests solving the face equations first.
A separated sector has vertices \(O,A_i,A_{i+1},U_i,W_i,B_i,C_i\)
and the seven triangles
\begin{equation}\label{eq:seven-design}
\begin{gathered}
 (O,A_i,U_i),\ (O,U_i,A_{i+1}),\
 (W_i,A_i,B_i),\ (W_i,B_i,C_i),\\
 (W_i,C_i,A_{i+1}),\
 (W_i,A_{i+1},U_i),\ (W_i,U_i,A_i).
\end{gathered}
\end{equation}
Putting \(\xi_i=r_{OA_i}\), \(\eta_i=r_{OU_i}\), all face equations
are solved by
\begin{equation}\label{eq:sector-weights}
\begin{aligned}
 r_{U_iA_i}&=-\xi_i-\eta_i,&r_{U_iA_{i+1}}&=-\xi_{i+1}-\eta_i,\\
 r_{U_iW_i}&=\eta_i+\tfrac12(\xi_i+\xi_{i+1}),\\
 r_{W_iA_i}=r_{W_iC_i}&=\tfrac12(\xi_i-\xi_{i+1}),&
 r_{W_iA_{i+1}}=r_{W_iB_i}&=\tfrac12(\xi_{i+1}-\xi_i).
\end{aligned}
\end{equation}
Conversely these formulas exhaust the solutions. They follow successively
from the three outer triangles, the two interface triangles and the two
triangles incident to \(O\).

At \(B_i,C_i\) there is just one interior edge. A nonzero edge weight
can satisfy the boundary gradient equation only if its endpoint
coefficient vanishes. This motivates placing
\[
 A_i,W_i,C_i\ \hbox{on one line},\qquad
 A_{i+1},W_i,B_i\ \hbox{on another}.
\]
Thus \(W_i\) is the diagonal intersection of the outer quadrilateral.
Putting \(U_i\) on the chord \(A_iA_{i+1}\) also gives
\(a_{O,U_i}=a_{W_i,U_i}=0\).
These choices are design decisions that simplify the equations; they
are not necessary for every possible obstruction.

In local affine coordinates
\[
 O=0,\quad A_i=(1,0),\quad A_{i+1}=(0,1),\quad
 U_i=(a,b),\quad a+b=1,\quad W_i=(x,y),
\]
assume \(a,b>0\), \(x+y>1\) and \(ay-bx\ne0\).
Substitution in the complete endpoint equations gives
\[
 (G_{U_i})_x+(G_{U_i})_y=-\frac{\xi_i+\xi_{i+1}}{2ab}.
\]
Consequently
\begin{equation}\label{eq:transfer-design}
 \xi_{i+1}=-\xi_i,\qquad
 \eta_i=-\xi_i\frac{bx+ay}{2(ay-bx)}.
\end{equation}
The coefficient determinant in this two-equation elimination is
\[
 \det\frac{\partial G_{U_i}}{\partial(\xi_{i+1},\eta_i)}
 =\frac{ay-bx}{a^2b^2(x+y-1)}\ne0.
\]
There is no missing zero-pivot branch in the stated geometry.
Odd numbers of sectors cannot carry the resulting nonzero alternating
state. Even numbers may, but the balances at the \(A_i\) and \(O\) remain
to be enforced.

\subsection{A complete geometric criterion}
We now allow \(n\ge3\) sectors, with arbitrary nonsymmetric geometry.
The central polygon is star-shaped with respect to \(O\), and
\[
 d_i=[O,A_i,A_{i+1}]>0,\qquad
 U_i=\alpha_iA_i+\beta_iA_{i+1},\quad
 \alpha_i,\beta_i>0,\quad\alpha_i+\beta_i=1.
\]
Each outer quadrilateral is strictly convex; \(W_i\) is its interior
diagonal intersection, lies across the chord from \(O\), and is not
collinear with \(O,U_i\).
The two outer corners adjacent to \(A_i\) may remain separate or be
identified to a single \(P_i\), producing the radial edge \(A_iP_i\).
At every such identification require
\begin{equation}\label{eq:radial-nonzero}
 [A_i,W_{i-1},W_i]\ne0.
\end{equation}
All sectors together must form a conforming triangulated disk.
Let \(k\) be the number of identified radial seams.

Define
\[
 K_i=[A_{i-1},A_i,A_{i+1}],\qquad
 D_i=[O,A_{i-1},A_{i+1}],\qquad
 \ell_i(X)=\frac{[A_i,A_{i+1},X]}{d_i}.
\]
When \(K_iD_i\ne0\), put
\(\gamma_i=D_i/K_i\) and
\[
 Z_i=O+\tfrac23\beta_i\gamma_i(A_{i+1}-A_i).
\]
All indices in this subsection are modulo \(n\).

\begin{theorem}[Sector-completion criterion]\label{thm:geometric}
Every mesh just described has
\[
 V=5n+1-k,\quad V_B=3n-2k,\quad V_I=2n+1+k,\quad
 N=7n,\quad E_I=9n+k,
\]
and
\begin{equation}\label{eq:completion-dimension}
 \dim S_3^1=13n+3-4k=3V_B+2V_I+1.
\end{equation}
Its vertex defect is either zero or one.
It is zero if \(n\) is odd or some \(K_iD_i=0\).
In the remaining case the defect is one if and only if
\begin{equation}\label{eq:geometric-three}
\begin{aligned}
 \alpha_i\gamma_{i+1}+\beta_i\gamma_i&=0 &&\text{for every }i,\\
 U_i,W_i,Z_i&\text{ are collinear} &&\text{for every }i,\\
 \sum_i(-1)^i\ell_i&\equiv0.&&
\end{aligned}
\end{equation}
\end{theorem}
\begin{proof}
At an identified seam, the endpoint coefficients at \(P_i\) on the
two \(P_iW\) edges vanish by diagonal collinearity. Its remaining balance is
\[
 a_{P_i,A_i}r_{A_iP_i}(A_i-P_i)=0.
\]
Condition \eqref{eq:radial-nonzero} forces \(r_{A_iP_i}=0\).
Thus all face-balanced weights, including the seam variables, reduce to
\eqref{eq:sector-weights}. Formula \eqref{eq:transfer-design} then leaves
at most the scalar state \(\xi_i=(-1)^ic\).
The \(W_i\) and outer-corner balances and loadings vanish identically.

Set \(v_i=A_i-O\) and
\[
 \kappa_i=-\frac{[O,U_i,W_i]}{[A_i,A_{i+1},W_i]}.
\]
For the nonzero candidate \(\xi_i=(-1)^i\), direct addition of all
incident terms, including the radial \(OA_i\) term, gives
\begin{equation}\label{eq:A-balance}
 \frac{G_{A_i}}{\xi_i}=
 \frac{D_i}{d_{i-1}d_i}v_i+
 \frac{3\kappa_{i-1}}{2\alpha_{i-1}d_{i-1}}(v_{i-1}-v_i)+
 \frac{3\kappa_i}{2\beta_i d_i}(v_i-v_{i+1}).
\end{equation}
Taking determinants with the two adjacent chord directions shows, for
\(K_i\ne0\), that this vanishes precisely when
\[
 \tfrac32\kappa_{i-1}=\alpha_{i-1}\gamma_i,\qquad
 \tfrac32\kappa_i=-\beta_i\gamma_i.
\]
Together over all indices these are the first two conditions in
\eqref{eq:geometric-three}. The second equivalence uses
\(\kappa_i=-2\beta_i\gamma_i/3\).

The zero cases must be handled before division. If \(K_i=0\), then
\(D_i=d_{i-1}+d_i>0\), and the determinant of
\eqref{eq:A-balance} with \(v_{i-1}-v_i\) is \(-D_i/d_i\ne0\).
If \(D_i=0\), the equations instead force \(\kappa_i=0\), contrary to
the hypothesis on \(O,U_i,W_i\). Neither case admits a nonzero candidate.

Finally \(a_{O,U_i}=0\), and summation at the center yields
\[
 G_O=-\sum_i\xi_i\frac{K_i}{d_{i-1}d_i}v_i
     =-2\sum_i\xi_i\frac{v_i-v_{i+1}}{d_i}.
\]
This is exactly the last condition of \eqref{eq:geometric-three}.
Every original face and vertex has now been included.

The loading at \(U_i\) is
\begin{equation}\label{eq:U-loading-geometric}
 R_{U_i}=\frac{(-1)^i\kappa_i}{2\alpha_i^2\beta_i^2d_i}\ne0.
\end{equation}
Hence a nonzero balanced candidate always has nonzero loading.
The full matrix \(H\) therefore has no left kernel, whether the mesh
is interpolatory or defective. The dimension and defect statements follow
from \cref{prop:duality} and the stated counts.
\end{proof}

For completeness, the other coefficients of the unique relation, with
the normalization just used, are
\begin{align}
 R_O&=-\sum_i(-1)^i\frac{K_i}{d_{i-1}d_i},\label{eq:complete-load}\\
 R_{A_i}&=(-1)^i\left[
 -\frac{D_i}{d_{i-1}d_i}
 -\frac{(1+2\beta_i)\kappa_i}{2\beta_i^2d_i}
 +\frac{(1+2\alpha_{i-1})\kappa_{i-1}}{2\alpha_{i-1}^2d_{i-1}}
 \right].\nonumber
\end{align}
All outer loadings vanish. Some \(A_i\) or center coefficients can vanish
in a general member of the family; only the \(U_i\) coefficients are
automatically nonzero. Thus the whole core need not always be a circuit.

\subsection{Reverse construction and changes of the outer completion}
The criterion determines the only possible chord point and outer line
from the central polygon:
\[
 \alpha_i=\frac{\gamma_i}{\gamma_i-\gamma_{i+1}},\qquad
 \beta_i=\frac{-\gamma_{i+1}}{\gamma_i-\gamma_{i+1}},\qquad
 U_i=\alpha_iA_i+\beta_iA_{i+1},\qquad L_i=U_iZ_i.
\]
It requires even \(n\), alternating signs of \(\gamma_i\), and central
closure. Set
\[
 W_i(\varepsilon)=U_i+\varepsilon(U_i-Z_i),\qquad
 P_i(\varepsilon)=
 \operatorname{line}(A_{i-1},W_{i-1})
 \cap\operatorname{line}(A_{i+1},W_i).
\]
For all sufficiently small \(\varepsilon>0\), this gives a legal
all-shared completion. Indeed, \(P_i(\varepsilon)\to A_i\);
the two diagonals meet strictly inside the resulting quadrilateral;
and the boundary is a small perturbation of the simple central polygon.
All face orientations are positive. The oriented boundary chain has
winding number one inside, so the positive local degrees exclude overlap.
Moreover
\[
 [A_i,W_{i-1},W_i]\longrightarrow
 -\alpha_{i-1}\beta_iK_i\ne0,
\]
which checks the seam condition.

For every \(n=2m\), \(m\ge3\), an explicit central polygon is obtained by
\[
 A_{2j}=(\cos 2j\phi,\sin 2j\phi),\quad
 A_{2j+1}=\rho(\cos(2j+1)\phi,\sin(2j+1)\phi),
 \quad \phi=\pi/m,\quad0<\rho<\cos\phi.
\]
Then
\[
 \gamma_{2j}=\frac{\rho\cos\phi}{1-\rho\cos\phi}>0,\qquad
 \gamma_{2j+1}=\frac{\cos\phi}{\rho-\cos\phi}<0.
\]
Each parity class of normalized side equations is a rotation orbit;
their linear coefficient sums vanish, and their constant sums agree.
Thus the central closure condition holds and the thin completion
constructs an actual defect-one mesh.

Changing the outer quadrilaterals, sliding \(W_i\) on the same line \(U_iW_i\),
or identifying a different permitted set of radial seams preserves the
core relation. All outer data remain free. Splitting one shared corner
into two diagonal corners inside its former quadrilaterals opens a boundary
notch, changes
\[
 (\Delta V,\Delta V_B,\Delta V_I,\Delta\dim S_3^1)=(1,2,-1,4),
\]
and keeps the defect. In particular the 25-vertex example has legal
26- through 31-vertex completions with the same thirteen-point relation.
This compares complete legal meshes; it does not justify arbitrary
local replacements inside an unrelated mesh.

\begin{corollary}[Positive cases and family-specific minimality]
\label{cor:completion-positive}
Every allowed completion with at most one reflex central corner is
interpolatory. Every completion with \(n\le5\) is interpolatory.
The least defective all-shared and all-separated sizes in this
restricted geometric family are respectively 25 and 31.
\end{corollary}
\begin{proof}
In a bad \(n=2m\) polygon the \(\gamma_i\) alternate.
A negative \(\gamma_i\) comes either from a reflex corner or from two
adjacent central angles whose sum exceeds \(\pi\).
On the negative parity class those angle windows partition the total angle
\(2\pi\); at most one can exceed \(\pi\).
Hence there are at least \(m-1\) reflex corners.
For \(n\ge6\) this excludes at most one reflex corner.
Odd \(n\) is already excluded.
For \(n=4\), closure gives \(\ell_0+\ell_2=\ell_1+\ell_3\).
The nonzero affine function \(\ell_2-\ell_3\) vanishes at
\(O,A_1,A_3\), forcing \(D_0=0\), a contradiction.
The six-sector examples below attain the two bounds.
\end{proof}

\subsection{Why four unrestricted sectors can work}
The 17-vertex example uses four shared sectors but abandons the chord and
diagonal restrictions. Write \(b_i=r_{A_iP_i}\).
All face-balanced weights are then parametrized by \(\xi_i,\eta_i,b_i\):
\begin{equation}\label{eq:general-shared-weights}
\begin{aligned}
 r_{U_iA_i}&=-\xi_i-\eta_i,&
 r_{U_iA_{i+1}}&=-\xi_{i+1}-\eta_i,\\
 L_i:=r_{W_iA_i}
 &=\tfrac12(\xi_i-\xi_{i+1}-b_i-b_{i+1}),\\
 r_{W_iA_{i+1}}
 &=\tfrac12(\xi_{i+1}-\xi_i-b_i-b_{i+1}),\\
 r_{U_iW_i}
 &=\eta_i+\tfrac12(\xi_i+\xi_{i+1}+b_i+b_{i+1}),\\
 r_{W_iP_i}&=-L_i-b_i,&r_{W_iP_{i+1}}&=L_i+b_i.
\end{aligned}
\end{equation}
Thus \(F=0\) is removed exactly before searching for geometry.
A reflection symmetry and one pressure normalization reduce \(G=0\)
to seventeen equations in seventeen unknowns.
A numerical candidate selects an isolating box; rational interval
inequalities certify its actual zero. The resulting \(b_i\) are nonzero,
and \(U_i\) are not constrained to their chords. It lies outside
\cref{thm:geometric}, so its four sectors do not contradict the restricted
six-sector minimum.
 \section{Four triangulations with a unique vertex relation}\label{sec:core}
We first recall the original rational 41-vertex mesh, which will also
serve as the block for the decomposable chains. The smaller examples
are then obtained from the design principles of Section~\ref{sec:design}.
Let $O=(0,0)$ and $Q(x,y)=(-y,x)$. Define the base coordinates by
\begin{equation}\label{eq:coordinates}
\begin{aligned}
p_{A_0}&=(1,0),&p_{A_1}&=(1/4,1/4),\\
p_{U_0}&=(5/14,3/14),&p_{U_1}&=(3/14,5/14),\\
p_{W_0}&=(11/28,19/84),&p_{W_1}&=(19/84,11/28),\\
p_{B_0}&=(5/12,2/9),&p_{B_1}&=(19/72,7/24),\\
p_{C_0}&=(7/24,19/72),&p_{C_1}&=(2/9,5/12).
\end{aligned}
\end{equation}
Set $p_{Z_{2j+t}}=Q^jp_{Z_t}$ for $Z\in\{A,U,W,B,C\}$, $t=0,1$,
$0\le j<4$. All indices are taken modulo eight.
For each $i$, the union of the seven triangles
\begin{equation}\label{eq:faces}
\begin{gathered}
(O,A_i,U_i),\ (O,U_i,A_{i+1}),\ (W_i,A_i,B_i),\ (W_i,B_i,C_i),\\
(W_i,C_i,A_{i+1}),\ (W_i,A_{i+1},U_i),\ (W_i,U_i,A_i).
\end{gathered}
\end{equation}
is called \emph{sector $i$}. The eight sectors form $\Dstar$, shown in
\cref{fig:core}. We call it the eight-armed triangulation, referring to its
eight interior paths $O-U_i-W_i$.

\begin{lemma}\label{lem:geometry}
$\Dstar$ is a conforming nondegenerate triangulated disk with
$(V,V_B,V_I,N,E_I)=(41,24,17,56,72)$.
Its boundary consists of the chains $A_iB_iC_iA_{i+1}$, its interior graph
is the union of the eight two-edge paths $O-U_i-W_i$, and $\sigma(\Dstar)=0$.
\end{lemma}
\begin{proof}
In sector zero, $U_0=\tfrac17A_0+\tfrac67A_1$,
$B_0=W_0+\tfrac16(W_0-A_1)$ and $C_0=W_0+\tfrac16(W_0-A_0)$.
Thus $W_0$ is the intersection of the diagonals $A_1B_0$ and $A_0C_0$ of the convex
quadrilateral $A_0B_0C_0A_1$, and lies on the side of the line $A_0A_1$ opposite to $O$:
$[A_0A_1O]=1/4>0>[A_0A_1W_0]=-1/56$.
The two triangles incident to $O$ partition the triangle $OA_0A_1$;
joining $W_0$ to $A_0,B_0,C_0,A_1,U_0$ partitions the quadrilateral into five
triangles. Strict interiority and the fact that $U_0$ subdivides the shared
edge $A_0A_1$ on both sides ensure nondegeneracy and conformity.
This sector lies in $0\le y\le x$, meeting the rays bounding it only along
the segments $OA_0$ and $OA_1$.
Reflection across $x=y$ and quarter turns about $O$ therefore yield a disk
with disjoint sector interiors. The face list yields the stated counts and
the interior graph.
The degrees of $O$ and the $W_i$ are $16$ and $5$; at $U_i$, the four
incident edges lie on three distinct lines---the chord $A_iA_{i+1}$ and the
lines $U_iO$, $U_iW_i$---so $U_i$ is not singular. Hence $\sigma=0$.
\end{proof}

\begin{figure}[tb]
\centering
\includegraphics[width=\linewidth,trim=0 4bp 0 0,clip]{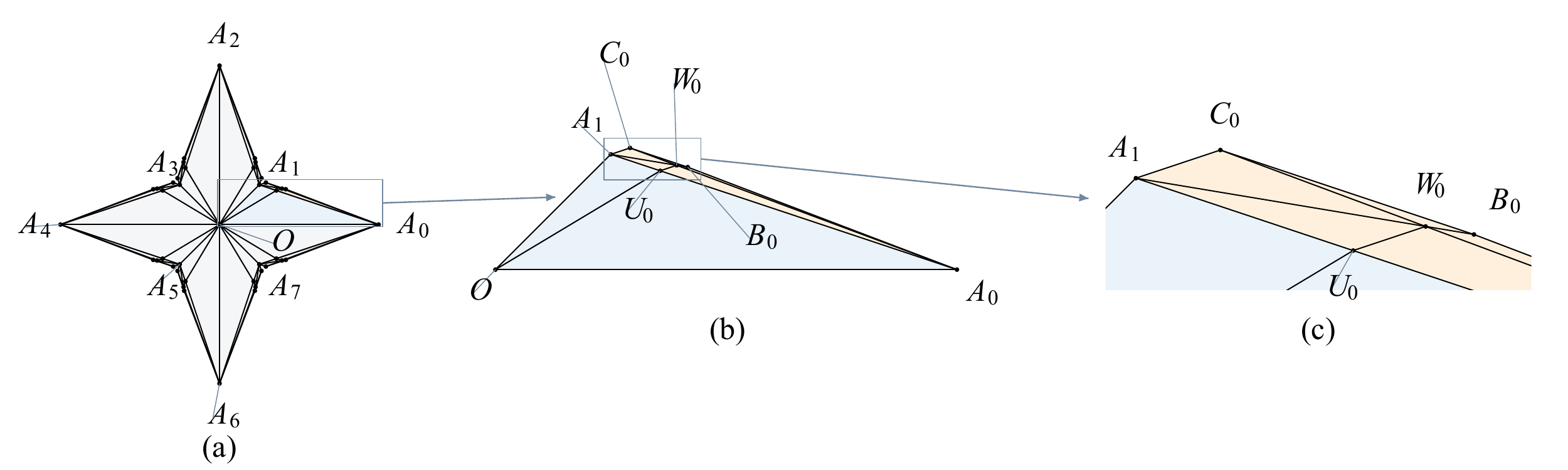}
 \caption{(a) The triangulation $\Dstar$; (b) sector zero; (c) detail near $U_0,W_0$.}\label{fig:core}
\end{figure}

Define the linear functional
\begin{equation}\label{eq:Lambda}
 \Lambda(f)=-756f_O+11\sum_{j=0}^3f_{A_{2j}}
       +864\sum_{j=0}^3f_{A_{2j+1}}-343\sum_{i=0}^7f_{U_i}.
\end{equation}
\begin{theorem}\label{thm:core}
For the triangulation $\Dstar$,
\[
 \im E_{\Dstar}=\ker\Lambda,\ \rank E_{\Dstar}=40,\
 \dim\Sspace(\Dstar)=107=3V_B+2V_I+1.
\]
The 17 evaluations in the support of $\Lambda$ are minimally dependent
(a circuit).
Deleting any one of these 17 vertices leaves the remaining 40 vertex
evaluations independent.
\end{theorem}
In particular, $f_O=1$ and $f_v=0$ for $v\ne O$ cannot be interpolated by any
spline in $\Sspace(\Dstar)$.
\begin{proof}
\textit{Step 1: Solve all face equations.}
Put $\xi_i=r_{OA_i}$, $\eta_i=r_{OU_i}$. Solving the seven face equations
$F_T(r)=0$ in sector $i$ gives
\begin{equation}\label{eq:face-param}
\begin{aligned}
r_{U_iA_i}&=-\xi_i-\eta_i,&r_{U_iA_{i+1}}&=-\eta_i-\xi_{i+1},\\
r_{U_iW_i}&=\eta_i+\tfrac12(\xi_i+\xi_{i+1}),\\
r_{W_iA_i}=r_{W_iC_i}&=\tfrac12(\xi_i-\xi_{i+1}),&
r_{W_iA_{i+1}}=r_{W_iB_i}&=\tfrac12(\xi_{i+1}-\xi_i).
\end{aligned}
\end{equation}
Conversely, these formulas satisfy all 56 face equations and give a linear
parametrization of their solution space by $(\xi,\eta)\in\R^{16}$.
The remaining equations defining $\ker H_0^{\mathsf T}$ are the vertex balances $G_v(r)=0$.

\textit{Step 2: Reduce balanced weights to one candidate.}
Substitution of \eqref{eq:face-param} and the coordinates into \eqref{eq:a}
and \eqref{eq:FGR} gives the balance at $U_0$:
\[
 (G_{U_0})_x=\tfrac76(31\xi_0-49\xi_1-8\eta_0),\
 (G_{U_0})_y=-\tfrac7{18}(45\xi_0-35\xi_1-8\eta_0).
\]
Clearing the scalar factors from the two equations gives $\xi_1=-\xi_0$ and $\eta_0=10\xi_0$.
Reflection across $x=y$ interchanges $U_0$ with $U_1$ and $A_0$ with $A_2$, fixing $A_1$;
hence $\xi_2=-\xi_1$, $\eta_1=-10\xi_1$.
Quarter turns yield $\xi_{i+1}=-\xi_i$, $\eta_i=10(-1)^i\xi_i$ for all $i$.
Hence every balanced weight is a multiple of the candidate $r^*$ with $\xi_i=(-1)^i,\ \eta_i=10.$

\textit{Step 3: Verify that the candidate satisfies every vertex balance.}
The $U_i$ balances hold by Step~2. Write $w_{ij}:=a_{ij}r^*_{ij}$;
direct substitution gives the following weights, listed in the order of the
interior-edge neighbors shown:
\begin{center}\small
\begin{tabular}{@{}c@{\qquad}l@{\quad}l@{}}
\toprule
Vertex&Neighbors&Weights\\\midrule
$A_0$&$O,U_0,W_0,U_7,W_7$&$-10/3,\ 308/27,\ -28/3,\ 308/27,\ -28/3$\\
$A_1$&$O,U_0,W_0,U_1,W_1$&$64,\ -336,\ 336,\ -336,\ 336$\\
$U_0$&$A_0,O,A_1,W_0$&$20482/27,\ -980/3,\ 4116,\ -13720/3$\\
$W_0$&$A_0,B_0,C_0,A_1,U_0$&$-392,\ 2352,\ -2352,\ 392,\ 0$\\
$B_0,C_0$&$W_0$&$0$\\\bottomrule
\end{tabular}
\end{center}
At $O$ the weights to $A_i$ with even and odd $i$ are $-6$ and $-8$,
respectively, and those to $U_i$ vanish. The vector contributions at $A_0,A_1,W_0$ are
\begin{align*}
G_{A_0}&=(10/3,0)+(-22/3,22/9)+(17/3,-19/9)
                 +(-22/3,-22/9)+(17/3,19/9)=0,\\
G_{A_1}&=(-16,-16)+(-36,12)+(48,-8)+(12,-36)+(-8,48)=0,\\
G_{W_0}&=(-238,266/3)+(56,-28/3)+(238,-266/3)+(-56,28/3)=0.
\end{align*}
In $G_O$, the contributions from opposite sectors cancel pairwise.
Quarter-turn rotations and the reflection across $x=y$ preserve $r^*$ and
transport these representative balances to all remaining vertices.
Together with Step~2, this proves $\ker H_0^{\mathsf T}=\R r^*$.

\textit{Step 4: Recover the image of $E_{\Dstar}$ and the dimension.}
Summing the endpoint weights in \eqref{eq:FGR} gives
$R_O=-56$, $R_{A_{2j}}=22/27$, $R_{A_{2j+1}}=64$,
$R_{U_i}=-686/27$, and $R_{W_i}=R_{B_i}=R_{C_i}=0$;
hence $R(r^*)\ne0$ and $\Lambda=(27/2)R(r^*)$.
Proposition~\ref{prop:duality} gives $\ker H^{\mathsf T}=0$ and
$(\im E_{\Dstar})^\circ=\R\Lambda$. Hence $\im E_{\Dstar}=\ker\Lambda$,
$\rank E_{\Dstar}=40$, and $\dim\Sspace=3\cdot41+56-72=107$.
Every relation is a multiple of $\Lambda$. Since its 17 coefficients are nonzero,
any relation whose support omits one of these vertices must have multiplier
zero, hence vanish.
This proves minimal dependence and the second assertion of the theorem.
\end{proof}
 \subsection{The rational six-sector examples}\label{sec:six-examples}
The 25- and 31-vertex meshes share the same thirteen-point core.
Set
\[
 Q_6=\begin{pmatrix}-1&-1/3\\3&0\end{pmatrix},\qquad
 J=\begin{pmatrix}-1&0\\3&1\end{pmatrix}.
\]
Then \(Q_6^3=I\), \(I+Q_6+Q_6^2=0\), \(J^2=I\) and
\(JQ_6J=Q_6^{-1}\).
They represent a threefold affine rotation and a reflection; they need
not appear as Euclidean isometries in these rational coordinates.
For every vertex class \(Z\) use
\[
 Z_{2j+t}=Q_6^j Z_t,\qquad 0\le j<3,\quad t=0,1.
\]
The common core is
\[
 O=0,\quad A_0=(1,0),\quad A_1=(0,1),\quad
 U_0=(1/16,15/16),\quad U_1=JU_0.
\]
All indices in this subsection are modulo six.

\paragraph{The 25-vertex mesh.}
Take
\[
 W_0=(1/15,43/45),\quad W_1=JW_0,\qquad
 P_0=(3/2,0),\quad P_1=(0,43/42).
\]
The vertices are \(O,A_i,U_i,W_i,P_i\); only the six \(P_i\) are
boundary vertices. In sector \(i\) use
\begin{equation}\label{eq:shared-faces}
\begin{gathered}
 (O,A_i,U_i),\ (O,U_i,A_{i+1}),\
 (W_i,A_i,P_i),\ (W_i,P_i,P_{i+1}),\\
 (W_i,P_{i+1},A_{i+1}),\
 (W_i,A_{i+1},U_i),\ (W_i,U_i,A_i).
\end{gathered}
\end{equation}
The convex outer quadrilateral has diagonal intersection \(W_i\).
For example
\[
 W_0=A_0+\tfrac{14}{15}(P_1-A_0)
     =P_0+\tfrac{43}{45}(A_1-P_0).
\]
The sectors lie in six disjoint angular cones, bounded by the rays
through
\[
 (1,0),(0,1),(-1,3),(-1/3,0),(0,-3),(1/3,-1).
\]
Their minimum doubled triangle area is \(1/720\).
Thus the construction is a conforming disk, rather than merely a list
of positively oriented triangles.

\begin{figure}[htbp]
\centering\includegraphics[width=\linewidth]{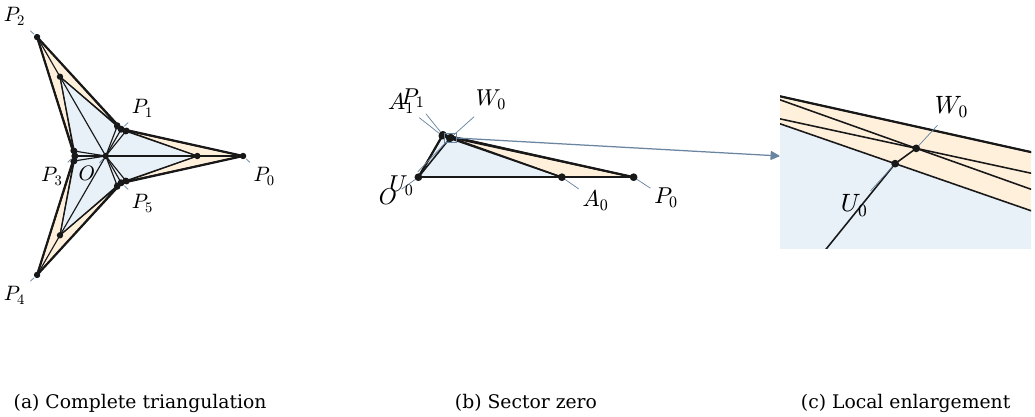}
\caption{The 25-vertex mesh, one sector, and the \(U_0,W_0\) detail.
For display only, all coordinates are transformed by the same invertible
linear map \(L(x,y)=(x+y/6,\sqrt3\,y/6)\).
The exact definition and every certificate use the rational coordinates.}
\label{fig:25}
\end{figure}

One can derive the fractions rather than guess them.
Initially put \(P_0=(p,0)\), \(P_1=(0,q)\),
\(U_0=(1-s,s)\), and let \(W_0\) be the diagonal intersection.
Choose \(\xi_i=(-1)^i\), \(\eta_i=\mu\), \(b_i=0\).
Writing \(D=(2pq-p-q)s+q(1-p)\), the remaining representative balance
numerators are
\begin{align*}
 N_U&=2\mu D+(p-q)s+q-pq,\\
 N_0&=5\mu D+(14pq-4p-9q-1)s-10pq+10q,\\
 N_1&=\mu D+(-6pq+5p+4q-3)s+4pq-3p-4q+3.
\end{align*}
Eliminating \(\mu D\) gives, near the displayed example,
\[
 q=\frac{13p+2}{12p+3},\qquad s=\frac{15}{16},\qquad
 \mu=\frac{14p+1}{2(p-1)}.
\]
For every \(p>1\) this is a legal defect-one family.
The choice \(p=3/2\) gives the displayed coordinates and \(\mu=22\).
This family moves boundary vertices; it does not describe the
fixed-boundary bad locus studied later.

\paragraph{The 31-vertex interior-tree mesh.}
Keep the same \(O,A_i,U_i\), but use
\[
 W_0=(1/7,9/7),\quad
 B_0=(8/49,65/49),\quad C_0=(1/49,72/49),
\]
and
\[
 W_1=JW_0,\qquad B_1=JC_0,\qquad C_1=JB_0.
\]
Use the separated faces \eqref{eq:seven-design}.
The boundary consists of \(A_i,B_i,C_i\); the interior graph consists
of the six paths \(O-U_i-W_i\). The diagonal identities
\[
 B_0=W_0+\tfrac17(W_0-A_1),\qquad
 C_0=W_0+\tfrac17(W_0-A_0)
\]
prove sector legality, and the angular-cone argument again excludes
overlap. The minimum doubled area is \(3/343\).

\begin{figure}[htbp]
\centering\includegraphics[width=\linewidth]{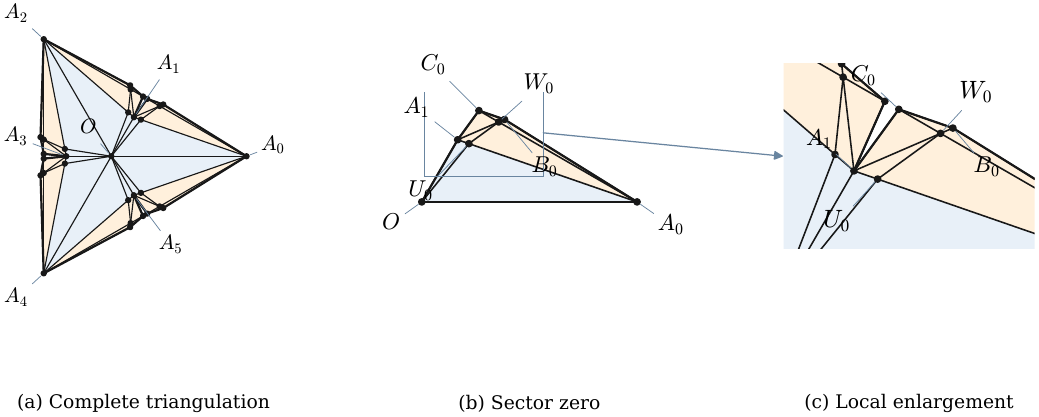}
\caption{The 31-vertex interior-tree mesh in the same three-panel layout.
The common affine display map is \(L\) from \cref{fig:25}. The extra outer
vertices preserve the interior tree while retaining the same core relation.}
\label{fig:31}
\end{figure}

To recover this construction, write \(U_0=(1-s,s)\), \(W_0=(h,k)\),
and retain the outer extension ratio \(1/7\). Choose
\(\xi_i=(-1)^i,\eta_i=m\), and put \(D=s(h+k)-k\).
The three nonautomatic representative numerators are
\begin{align*}
 N_U&=2mD-hs+k(s-1),\\
 N_A&=5mD+4hs+9ks-10k+s,\\
 N_D&=mD+(3-5s)h+4(1-s)k+3s-3.
\end{align*}
The combination
\[
 (2N_A-5N_U)+5(2N_D-N_U)
       =-2(16s-15)(h+k-1)
\]
forces \(s=15/16\) on the nondegenerate branch. The remaining equations give
\[
 k=\frac{13h+2}{3},\qquad m=\frac{29h+1}{32h-2}.
\]
Taking \(h=1/7\) gives \(m=2\) and the specified mesh.
Every \(h>1/8\) yields another legal example in this construction;
its boundary also moves.

\begin{theorem}[The common thirteen-point relation]\label{thm:six-examples}
For both displayed meshes the complete admissibility condition is
\begin{equation}\label{eq:Lambda-six}
 \Lambda_6(f)=
 -1800f_O+17\sum_{j=0}^2f_{A_{2j}}
 +8775\sum_{j=0}^2f_{A_{2j+1}}
 -4096\sum_{i=0}^5f_{U_i}=0.
\end{equation}
The evaluation ranks are 24 and 30, and the spline dimensions are 57 and 81,
respectively. The thirteen supported evaluations form a circuit.
\end{theorem}
\begin{proof}
For the 25-vertex mesh the complete boundary balances first force
\(b_i=0\). For its remaining arbitrary face-balanced weights,
\[
 G_{U_0}=\left(
 -\tfrac8{15}(8\eta_0-83\xi_0+93\xi_1),\
 \tfrac8{15}(8\eta_0-99\xi_0+77\xi_1)\right).
\]
Hence \(\xi_1=-\xi_0,\eta_0=22\xi_0\).
For the 31-vertex mesh the corresponding formula is
\[
 G_{U_0}=\left(
 -\tfrac8{15}(8\eta_0-3\xi_0+13\xi_1),\
 \tfrac8{15}(8\eta_0-19\xi_0-3\xi_1)\right),
\]
giving \(\xi_1=-\xi_0,\eta_0=2\xi_0\).
Applying the same local equations after the \(Q_6,J\) transformations
leaves only one global amplitude in either mesh.
The resulting candidates satisfy all other balances, either directly or
by \cref{thm:geometric}. In both cases their loadings are
\[
 R_O=-8,\quad R_{A_{2j}}=17/225,\quad
 R_{A_{2j+1}}=39,\quad R_{U_i}=-4096/225,
\]
with all outer loadings zero. Multiplication by 225 gives
\eqref{eq:Lambda-six}. The candidate loading is nonzero, so the full
matrix has no left kernel. Proposition~\ref{prop:duality} gives the
dimensions and ranks. Uniqueness and nonvanishing on the thirteen
supported vertices give the circuit assertion.
\end{proof}

\subsection{The certified 17-vertex example}\label{sec:example17}
Use the shared-sector face list \eqref{eq:shared-faces} with indices
modulo four. The boundary is the diamond
\[
 P_0=(1,0),\quad P_1=(0,1),\quad P_2=(-1,0),\quad P_3=(0,-1).
\]
The other thirteen vertices are internal.
Introduce geometry and pressure variables
\[
 x=(o,a_0,a_1,u_{0x},u_{0y},u_{1x},u_{1y},
 w_0,w_{1x},w_{1y},\xi_0,\xi_1,\xi_2,\eta_0,b_0,b_1,b_2).
\]
Set
\begin{align*}
 O&=(o,0),&A_0&=(a_0,0),&A_2&=(-2/5,0),\\
 A_1&=(2/75,a_1),&A_3&=(2/75,-a_1),\\
 U_0&=(u_{0x},u_{0y}),&U_3&=(u_{0x},-u_{0y}),\\
 U_1&=(u_{1x},u_{1y}),&U_2&=(u_{1x},-u_{1y}),\\
 W_0&=(2/5,w_0),&W_3&=(2/5,-w_0),\\
 W_1&=(w_{1x},w_{1y}),&W_2&=(w_{1x},-w_{1y}),
\end{align*}
and in \eqref{eq:general-shared-weights} use
\[
 \xi=(\xi_0,\xi_1,\xi_2,\xi_1),\quad
 \eta=(\eta_0,1,1,\eta_0),\quad
 b=(b_0,b_1,b_2,b_1).
\]
We define the example by seventeen exact rational balance equations.
For compactness write
\[
 \alpha=\frac{2}{75},\qquad \gamma=-\frac25,\qquad \kappa=\frac25,
 \qquad \omega_{vw}(x)=a_{vw}(x)r_{vw}(x).
\]
The endpoint coefficients \(a_{vw}\) are the explicit area quotients
\eqref{eq:a}, and every \(r_{vw}\) is specified by
\eqref{eq:general-shared-weights}. Thus \(\omega_{vw}\) are rational
functions of the seventeen components of \(x\), not additional unknowns.
The five axis equations are
\begin{subequations}\label{eq:17-defining-system}
\begin{align}
0={}&\begin{aligned}[t]
 &\omega_{OA_0}(a_0-o)+2\omega_{OA_1}(\alpha-o)
       +\omega_{OA_2}(\gamma-o)\\
 &\quad+2\omega_{OU_0}(u_{0x}-o)+2\omega_{OU_1}(u_{1x}-o),
\end{aligned}\\
0={}&\omega_{A_0O}(o-a_0)+2\omega_{A_0U_0}(u_{0x}-a_0)
       +2\omega_{A_0W_0}(\kappa-a_0)+\omega_{A_0P_0}(1-a_0),\\
0={}&\omega_{A_2O}(o-\gamma)+2\omega_{A_2U_1}(u_{1x}-\gamma)
       +2\omega_{A_2W_1}(w_{1x}-\gamma)+\omega_{A_2P_2}(-1-\gamma),\\
0={}&\omega_{P_0A_0}(a_0-1)+2\omega_{P_0W_0}(\kappa-1),\\
0={}&\omega_{P_2A_2}(\gamma+1)+2\omega_{P_2W_1}(w_{1x}+1).
\end{align}
The remaining twelve scalar equations are the two components of each of
the following six vector equations:
\begin{align}
\binom00={}&\begin{aligned}[t]
 &\omega_{A_1O}\binom{o-\alpha}{-a_1}
 +\omega_{A_1U_0}\binom{u_{0x}-\alpha}{u_{0y}-a_1}
 +\omega_{A_1U_1}\binom{u_{1x}-\alpha}{u_{1y}-a_1}\\
 &+\omega_{A_1W_0}\binom{\kappa-\alpha}{w_0-a_1}
 +\omega_{A_1W_1}\binom{w_{1x}-\alpha}{w_{1y}-a_1}
 +\omega_{A_1P_1}\binom{-\alpha}{1-a_1},
\end{aligned}\\[2pt]
\binom00={}&
 \omega_{U_0O}\binom{o-u_{0x}}{-u_{0y}}
 +\omega_{U_0A_0}\binom{a_0-u_{0x}}{-u_{0y}}
 +\omega_{U_0A_1}\binom{\alpha-u_{0x}}{a_1-u_{0y}}
 +\omega_{U_0W_0}\binom{\kappa-u_{0x}}{w_0-u_{0y}},\\[2pt]
\binom00={}&
 \omega_{U_1O}\binom{o-u_{1x}}{-u_{1y}}
 +\omega_{U_1A_1}\binom{\alpha-u_{1x}}{a_1-u_{1y}}
 +\omega_{U_1A_2}\binom{\gamma-u_{1x}}{-u_{1y}}
 +\omega_{U_1W_1}\binom{w_{1x}-u_{1x}}{w_{1y}-u_{1y}},\\[2pt]
\binom00={}&\begin{aligned}[t]
 &\omega_{W_0A_0}\binom{a_0-\kappa}{-w_0}
 +\omega_{W_0P_0}\binom{1-\kappa}{-w_0}
 +\omega_{W_0P_1}\binom{-\kappa}{1-w_0}\\
 &+\omega_{W_0A_1}\binom{\alpha-\kappa}{a_1-w_0}
 +\omega_{W_0U_0}\binom{u_{0x}-\kappa}{u_{0y}-w_0},
\end{aligned}\\[2pt]
\binom00={}&\begin{aligned}[t]
 &\omega_{W_1A_1}\binom{\alpha-w_{1x}}{a_1-w_{1y}}
 +\omega_{W_1P_1}\binom{-w_{1x}}{1-w_{1y}}
 +\omega_{W_1P_2}\binom{-1-w_{1x}}{-w_{1y}}\\
 &+\omega_{W_1A_2}\binom{\gamma-w_{1x}}{-w_{1y}}
 +\omega_{W_1U_1}\binom{u_{1x}-w_{1x}}{u_{1y}-w_{1y}},
\end{aligned}\\[2pt]
\binom00={}&
 \omega_{P_1A_1}\binom{\alpha}{a_1-1}
 +\omega_{P_1W_0}\binom{\kappa}{w_0-1}
 +\omega_{P_1W_1}\binom{w_{1x}}{w_{1y}-1}.
\end{align}
\end{subequations}
The first five equations are the axis balances.
The six vector equations give the balances at
\(A_1,U_0,U_1,W_0,W_1,P_1\), in that order.
Reflection supplies the omitted components and vertices.
Let \(g(x)\in\R^{17}\) collect the scalar components on the right-hand sides.
The exact implicit definition is
\begin{equation}\label{eq:17-implicit-root}
 \boxed{x^*=\text{the unique real zero of }g(x)=0\text{ in }X_{17}.}
\end{equation}
The vertex coordinates and edge weights are obtained by substituting
\(x^*\) into the formulas above.
Here \(X_{17}\) is the certified rational isolating box supplied in
\texttt{example17-implicit.json}, which also contains the seventeen
expressions with all area quotients substituted.
The box selects the real branch and is not a numerical assignment of
the coordinates. Its uniqueness certificate is recalled below.

\begin{theorem}[Seventeen vertices]\label{thm:17}
The exact implicit system \eqref{eq:17-defining-system} has exactly one
zero in \(X_{17}\). Its geometry is a conforming nonsingular triangulated disk
with \((V,V_B,V_I,N,E_I)=(17,4,13,28,40)\), and
\[
 \dim S_3^1=39,\qquad \rank E=16,\qquad \dd=1.
\]
The unique relation is \(R(r^*)^{\mathsf T}f=0\), where the pressure
\(r^*\) is evaluated at this exact zero. Every one of its seventeen
coefficients is nonzero.
\end{theorem}
\begin{proof}
Let \(x_0\) be the rational center of \(X=X_{17}\).
The rational certificate gives a nonsingular rational preconditioner \(C\)
and an interval Jacobian \(J(X)\) such that, componentwise,
\begin{equation}\label{eq:interval-contraction}
 |Cg(x_0)|+|I-CJ(X)|\,\operatorname{rad}(X)
 <\operatorname{rad}(X),\qquad
 \norm{I-CJ(X)}_\infty<1.
\end{equation}
Thus \(x\mapsto x-Cg(x)\) maps \(X\) strictly into itself and is a
contraction. It has one fixed point, which is precisely a zero of \(g\).
The inclusion ratio is less than \(1/20000\), and the contraction
norm is less than \(1/80000\). These are strict rational bounds
recorded by the certificate; no approximate zero is used in the argument.

All 28 doubled areas have rational interval lower bound greater than \(1/50\).
The abstract complex is a disk with the stated simple diamond boundary.
The sum of positive local triangle degrees is its boundary winding number,
one inside and zero outside. This excludes overlapping interiors and
establishes the required embedding. The interior vertices of degree four are
certified nonsingular.

The face matrix has rank 28, leaving twelve face-balanced parameters.
On their exact basis, an \(11\times11\) minor of the gradient-balance
matrix is uniformly nonsingular on \(X\), with rational Neumann bound
less than \(1/10^7\). The constructed nonzero pressure therefore
spans the whole balanced kernel. Its loading satisfies
\[
 -\frac35<R_O<-\frac12.
\]
Thus the loading is nonzero, \(H\) has full row rank 40 and \(H_0\)
has rank 39. The dimension is \(3\cdot17+28-40=39\).
Finally interval bounds separate every loading coefficient from zero,
so the supported evaluations form a seventeen-element circuit.
\end{proof}

\begin{figure}[htbp]
\centering\includegraphics[width=\linewidth]{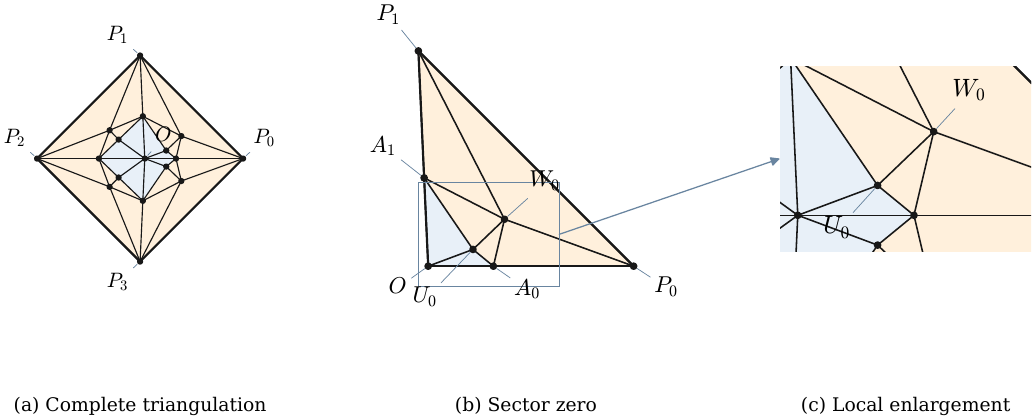}
\caption{The 17-vertex example, sector zero, and its local detail.
The display uses rational isolating-box centers; the exact geometry is
the certified zero of \eqref{eq:17-defining-system}.
Unlike the separated-sector examples, \(U_i\) need not lie on a chord
and all eighty interior-edge endpoint coefficients are nonzero.}
\label{fig:17}
\end{figure}
For every example in this section and \cref{sec:core},
the center impulse \(f_O=1\), \(f_v=0\) otherwise, is inadmissible.
The distinction between total vertex count and relation support is essential:
the old 41-vertex relation already had seventeen supported evaluations,
but it was not a seventeen-vertex mesh.
 \section{Fixed-boundary bad loci and interpolation stability}\label{sec:stability}
\subsection{A common local reduction}
For each of the four examples, fix the complete labelled boundary and
face list. Let \(\theta\in\R^{2V_I}\) collect every interior coordinate,
without a symmetry restriction or a quotient by coordinate transformations.
Let \(\theta_*\) denote the specified example.
In a sufficiently small neighborhood all triangles remain nondegenerate
and the disk embedding remains valid. The full matrix \(H\) has constant
full row rank there, because it has full row rank at \(\theta_*\).
The spline dimension is consequently constant.

The face equations do not depend on geometry. Choose an exact basis
\(B_F\) for their kernel and write
\[
 r=B_Fz,\qquad M(\theta)z=G(B_Fz)=0.
\]
At the base point this kernel has dimension one. Normalize a nonzero
coordinate of \(z\), and use \(d-1\) independent equations, where
\(d=\dim\ker F\), to solve the other pressure coordinates:
\begin{equation}\label{eq:pressure-candidate}
 z(\theta)=(1,-P_M(\theta)^{-1}b_M(\theta)).
\end{equation}
Substituting in the remaining equations defines an exact rational vector
\(\Psi(\theta)\). Locally,
\begin{equation}\label{eq:bad-set-exact}
 \Sigma=\{\theta:\rank E_\theta<V\}
       =\{\theta:\Psi(\theta)=0\}.
\end{equation}
On \(\Sigma\) the defect is exactly one; outside it evaluation is
surjective. This reduction includes nonsymmetric perturbations.

\begin{table}[ht]
\centering\small
\caption{Sizes and exact local linearization ranks. The joint variables
are all geometry coordinates and all normalized pressure coordinates.}
\label{tab:joint-ranks}
\begin{tabular}{@{}rrrrrr@{}}
\toprule
\(V\)&\(2V_I\)&\(d\)&size of \(M\)&joint Jacobian&joint rank\\\midrule
17&26&12&\(34\times12\)&\(34\times37\)&33\\
25&38&18&\(50\times18\)&\(50\times55\)&48\\
31&26&12&\(62\times12\)&\(62\times37\)&37\\
41&34&16&\(82\times16\)&\(82\times49\)&49\\\bottomrule
\end{tabular}
\end{table}
The rational examples use exact rational minors. The 17-vertex rank
lower bound uses a rational interval minor; an actual four-parameter
family supplies the matching upper bound. A kernel of a linearization
alone would not determine the dimension of the analytic bad set.

\subsection{The Schur residual and the cost of arbitrary data}
Write \(m=E_I\). Fix a nonsingular \((m-1)\)-square submatrix of \(H_0\)
at the base point. After row and column permutations put
\[
 H_0=\begin{pmatrix}P&Q\\r&s\end{pmatrix},\qquad
 H_f=\begin{pmatrix}B\\c\end{pmatrix},
\]
and define
\begin{equation}\label{eq:Schur}
 S=s-rP^{-1}Q,\qquad a=c-rP^{-1}B,\qquad
 \sigma(\theta)=\norm{S(\theta)}_2.
\end{equation}
Here \(S\) is a single row. The remaining interpolation equation is
\[
 S(\theta)u_2=-a(\theta)f.
\]
The full row rank of \(H\) implies \(a(\theta_*)\ne0\).
All elimination matrices and their inverses are uniformly bounded on a
fixed smaller neighborhood. In particular,
\begin{equation}\label{eq:equivalent-residuals}
 \sigma_{\min}(H_0(\theta))\asymp \sigma(\theta)
 \asymp\norm{\Psi(\theta)},\qquad
 \Sigma=\{\sigma=0\}.
\end{equation}
The smallest singular value here is the smallest row singular value.

For a good configuration define
\[
 \mathcal C_\theta(f)=
 \inf\{\norm{s}_{L^\infty(\Omega)}:s\in S_3^1(\Delta(\theta)),\
 E_\theta s=f\}.
\]
\begin{theorem}[Uniform all-data estimate]\label{thm:uniform-cost}
There are constants \(c,C>0\) and a fixed neighborhood, independent of
the data \(f\), such that for every good configuration in that neighborhood
\begin{equation}\label{eq:uniform-cost}
 c\left(\norm{f}_2+\frac{|a(\theta)f|}{\sigma(\theta)}\right)
 \le\mathcal C_\theta(f)\le
 C\left(\norm{f}_2+\frac{|a(\theta)f|}{\sigma(\theta)}\right).
\end{equation}
At a bad configuration the exact compatibility condition is
\(a(\theta)f=0\), and compatible data have a uniformly bounded choice of
interpolant within the bad locus.
\end{theorem}
\begin{proof}
Every solution satisfies
\(|af|\le\sigma\norm{u_2}\).
Conversely, on the good set choose
\[
 u_2=-\frac{af}{\sigma^2}S^{\mathsf T},\qquad
 u_1=P^{-1}(-Bf-Qu_2).
\]
This solves all equations and gives the upper bound for the jet and
central-coefficient norm. If \(\sigma=0\) and \(af=0\), choose \(u_2=0\).
Uniform equivalence between these coordinates, the repeated local
Bernstein coefficients, and the supremum norm follows from
nondegeneracy and fixed-degree norm equivalence on a reference triangle.
Finally \(\norm{s}_\infty\ge\norm{f}_\infty\).
These estimates prove both bounds, after changing fixed constants.
\end{proof}
The row \(a(\theta)\) is a moving data functional. It is a relation
at bad points; it is an elimination row at good points. Replacing it by
one fixed relation is justified only when its variation can be absorbed
into the \(\norm{f}\) term at the relevant residual scale.

\subsection{The complete four-dimensional locus for seventeen vertices}
\begin{theorem}\label{thm:17-locus}
With the four diamond corners fixed, the entire nearby bad locus
\(\Sigma_{17}\subset\R^{26}\) is an embedded real analytic
four-dimensional manifold. It has codimension 22, and
\[
 \sigma(\theta)\asymp \dist(\theta,\Sigma_{17}).
\]
\end{theorem}
\begin{proof}
First allow the three reflection-symmetric anchors
\[
 x(A_1)=x(A_3),\qquad x(A_2),\qquad x(W_0)=x(W_3)
\]
to vary in the system defining the exact example.
The certified \(17\times17\) state Jacobian remains invertible,
giving a genuine analytic three-parameter family of
reflection-symmetric bad meshes.

A fourth parameter comes from a piecewise projective motion:
\begin{equation}\label{eq:piecewise-projective}
 \Phi_t(x,y)=
 \begin{cases}
 ((1-t)x/(1-t+ty),\,y/(1-t+ty)),&y\ge0,\\
 ((1+t)x/(1+t+ty),\,y/(1+t+ty)),&y\le0.
 \end{cases}
\end{equation}
For \(|t|<1\) both denominators are positive in the corresponding
half-diamond. Each half is mapped projectively onto itself; the diameter
is fixed pointwise. The four boundary vertices remain fixed.
One cannot infer \(C^1\) covariance by simply composing a spline
with this piecewise map, since its first derivative can jump at the seam.
Instead the full endpoint formulas verify pressure preservation.

For an ordinary projective map
\(p'=(Ap+b)/h(p)\), \(h(p)=c^{\mathsf T}p+d\), with homogeneous
matrix \(M\), the area identity yields
\begin{equation}\label{eq:projective-pressure}
 a'_{ij}=\frac{h_i^2h_j}{\det M}a_{ij},\quad
 G'_i=\frac{h_i^2}{\det M}(A-p'_ic^{\mathsf T})G_i,\quad
 R'_i=\frac{h_i^3R_i+h_i^2c^{\mathsf T}G_i}{\det M}.
\end{equation}
Keep all unoriented edge weights unchanged.
Off the diameter, a whole star lies in one half, so these formulas apply.
On the diameter write the upper and lower maps with parameters
\(u=t/(1-t)\), \(v=t/(1+t)\). The identity
\[
 \frac1{1+u}+\frac1{1-v}=2
\]
shows that both endpoint coefficients on an axis edge are unchanged.
For each reflected pair of off-axis neighbors, its summed gradient and
value contributions at an axis vertex are unchanged as well.
Thus all \(F,G\) equations and the nonzero center loading are preserved.

The new geometric derivative at \(t=0\) is
\[
 (-xy,y(1-y))\quad(y>0),\qquad
 (-xy,-y(1+y))\quad(y<0),
\]
and is zero on the diameter. It is reflection-odd, whereas the first
three tangent directions are reflection-even.
The four directions are independent; for example a \(4\times4\)
geometric parameter minor is \(-2x(U_0)y(U_0)\ne0\).
This constructs a four-dimensional embedded analytic family.

Conversely, the joint normalized balance system has 37 variables and
a certified nonsingular \(33\times33\) Jacobian minor.
Those 33 equations define a four-dimensional analytic container.
The constructed family lies inside it with full tangent rank, so by
the inverse function theorem it contains a relative neighborhood of
the base point. The full zero set equals this family locally.
Pressure normalization is uniquely and analytically solved by eleven
independent weight equations; the projection to geometry is an embedding.
This proves completeness, not only a lower bound on dimension.

After removing the eleven pressure directions, \(D\Psi\) has rank
\(33-11=22\), equal to the normal codimension.
Analytic normal coordinates and \eqref{eq:equivalent-residuals} give
the claimed distance comparison.
\end{proof}
Let \(\pi(\theta)\) be the nearest point on this local bad manifold.
The smooth variation of \(a\) and \cref{thm:uniform-cost} imply
\begin{equation}\label{eq:17-distance-cost}
 \mathcal C_\theta(f)\asymp\norm{f}_2+
 \frac{\dist_2(f,\ker a(\pi(\theta)))}
      {\dist(\theta,\Sigma_{17})}.
\end{equation}
In particular the center impulse has cost comparable to the reciprocal
of that distance. Distance to the single base point would be incorrect
along this positive-dimensional locus.

\subsection{The two analytic components for twenty-five vertices}
\begin{theorem}\label{thm:25-locus}
In the full fixed-boundary space \(\R^{38}\), there are real analytic
local coordinates
\[
 (q,b,c,\zeta)\in\R^{31}\times\R\times\R\times\R^5
\]
and an exactly defined analytic function \(\varphi(b,\zeta)\) such that
\begin{equation}\label{eq:25-normal-form}
 \Sigma_{25}=\{q=0,\ bc=0,\ b\varphi(b,\zeta)=0\}.
\end{equation}
There are precisely two analytic irreducible components through the
base point:
\[
 \mathcal B_0=\{q=0,b=0\},\qquad
 \mathcal B_1=\{q=0,c=0,\varphi=0\}.
\]
Their real dimensions are six and five, and their intersection has
dimension four. The first is smooth; the second and the whole bad locus
are singular at the base point.
Moreover
\begin{equation}\label{eq:25-scale}
 \sigma(\theta)\asymp
 \norm{q}+|b|\bigl(|c|+|\varphi(b,\zeta)|\bigr).
\end{equation}
\end{theorem}
\begin{proof}
Normalize \(\xi_0=1\) in the complete 18-parameter face kernel.
The joint system has 55 variables and 50 gradient components.
Its exact derivative rank is 48. Forty-eight independent components
therefore define a seven-dimensional analytic container.
The seventeen pressure variables are solved analytically; its projection
is a seven-dimensional geometric container.

The seam weight \(b_1\), or equivalently \(b_5\) times a nonzero analytic
unit, is a regular coordinate \(b\) there.
The locus of all seam weights zero is an actual six-dimensional smooth
branch: in the chord/diagonal parameterization, the fourteen remaining
center and \(A_i\) equations have a full-row-rank \(14\times20\)
geometric derivative.
Both residual equations on the seven-dimensional container vanish on
this branch and hence have a factor \(b\).
The first quotient has a nonzero differential; use it as a coordinate
\(c\). Subtracting its analytic multiple from the other quotient leaves
an exact function \(\varphi(b,\zeta)\). Together with 31 independent
normal equations this gives \eqref{eq:25-normal-form}.
The analytic function is defined by the fixed equations and pivot choices,
not by truncating a power series.

The first nonzero homogeneous term of \(\varphi\) is cubic.
Exact rational implicit differentiation, including the curvature
corrections of the 48 solved equations, gives
\[
 \varphi(m,b,E)=P_3(E)+O(\norm{(m,b,E)}^4),\qquad E\in\R^4,
\]
after an invertible linear change of the six remaining coordinates.
The certificate of its coefficients has a real binary-cubic discriminant
strictly greater than zero. A real linear change of the two complex
variables consequently reduces the leading term to
\begin{equation}\label{eq:25-cubic}
 P_3\sim\operatorname{Im}\{zw(z-w)\}.
\end{equation}
The six coordinates can be read to first order as the mean change of
\(\eta_i\), the seam weight, the two coordinates of \(O\), and the two
coordinates of \(\sum_{j=0}^2 A_{2j}\).
The full rational transformations and coefficients are supplied with
the certificate.

For \(H(z,w)=zw(z-w)\), the two derivatives
\(w(2z-w)\) and \(z(z-2w)\) have no common nonzero complex zero.
The complexification of \(\operatorname{Im}H\) is, up to a
nonzero factor, \(H(z,w)-H(Z,W)\); its projective cubic hypersurface
is nonsingular. A reducible cubic hypersurface in projective three-space
would have intersecting positive-degree components and be singular
there. Thus \(P_3\) is absolutely irreducible.
The lowest homogeneous terms of a factorization of \(\varphi\) would
factor \(P_3\), so \(\varphi\) is analytically irreducible.

The odd real cubic has nonzero real zero directions, all smooth.
Blowing up along such a direction and applying the implicit function
theorem produces actual regular zeros of \(\varphi\), including ones
with \(b\ne0\), arbitrarily near the base point.
It gives real dimension five for \(\mathcal B_1\), and, with \(b=0\),
dimension four for the intersection. The cubic zero cone is not linear,
so the secondary component is not a \(C^1\) submanifold at the base
point. The exact normal form gives no additional components.

Finally the residual equations can be changed by bounded invertible
analytic row operations into \(q,bc,b\varphi\).
Terms off the container are analytic multiples of \(q\) and are removed
by row subtraction. No division by \(b\) is used in this norm comparison.
Equation \eqref{eq:equivalent-residuals} now proves \eqref{eq:25-scale}.
\end{proof}

Combining \eqref{eq:25-scale} with \eqref{eq:uniform-cost} gives the
correct full-neighborhood estimate. It cannot in general be replaced
by one fixed power of \(d=\dist(\theta,\Sigma_{25})\).
There are legal good paths on which
\[
 (\sigma,d)\asymp(|t|,|t|),\qquad
 (|t|^2,|t|),\qquad (|t|^4,|t|).
\]
For the second, a normalized kernel direction has \(b\asymp t\),
\(c\asymp t\) and a certified nonzero quadratic obstruction.
For the third take \(q=c=0\) and, in the physical coordinates above,
\((m,b,x,y,X,Y)=(0,t,t,t,0,0)\); its cubic coefficient is nonzero.
Both directions lie outside the corresponding bad tangent cone,
which gives \(d\asymp|t|\).
The center-impulse costs therefore have orders \(d^{-1},d^{-2},d^{-4}\).
This does not identify a worst possible universal exponent.
The one-sided bound \(\sigma\le C d\) always holds by local
Lipschitz continuity.

\subsection{The transverse isolated tree examples}
\begin{theorem}\label{thm:isolation}
For the 31- and 41-vertex examples, with all 18 or 24 boundary vertices
fixed, the only nearby bad geometry is \(\theta_*\).
Every sufficiently small nonzero internal perturbation restores
surjectivity, while the respective spline dimensions remain 81 and 107.
Furthermore, uniformly for all data at good configurations,
\begin{equation}\label{eq:tree-cost}
 \mathcal C_\theta(f)\asymp\norm{f}_2+
 \frac{\dist_2(f,\ker\Lambda_*)}{\norm{\theta-\theta_*}_2}.
\end{equation}
\end{theorem}
\begin{proof}
For 31 vertices, the joint \(62\times37\) derivative has rank 37.
There is also a short readable elimination.
The boundary \(B_i,C_i\) balances first fix all \(W_i\).
The \(W_i\) equations force the \(U_i\) variations to stay on their
fixed chords. The \(U_i\) normal equations fix alternating \(\xi_i\);
their tangential equations solve \(\eta_i\).
The six even-\(A_i\) components solve the chord parameters as functions
of \(O=(x,y)\). Two remaining odd-\(A_i\) equations have derivative
\[
 \begin{pmatrix}-80&0\\0&40/3\end{pmatrix}.
\]
It is nonsingular, so every normalized joint infinitesimal variation
is zero.

For 41 vertices the same elimination gives rank 49 for the
\(82\times49\) joint derivative.
The even-\(A_i\) local chord blocks are
\[
 \frac{931}{36}\begin{pmatrix}-3&3\\1&1\end{pmatrix},
\]
and the final two equations for the center have derivative
\[
 \frac1{19}\begin{pmatrix}-512&384\\384&-512\end{pmatrix},
 \qquad \det=\frac{114688}{361}\ne0.
\]
These exact reductions include every internal coordinate and normalized
pressure variable. Subtracting 11 or 15 independent pressure directions
gives geometric residual ranks 26 and 34.

Thus \(D\Psi(\theta_*)\) is injective in each full geometric space.
A suitable square subsystem has a locally unique zero, and
\(\norm{\Psi(\theta)}\asymp\norm{\theta-\theta_*}\).
The full bad set is contained in that zero set and contains \(\theta_*\);
it is precisely the single point.
Equations \eqref{eq:equivalent-residuals} and
\eqref{eq:uniform-cost} prove the estimate with the moving row.
Scale that row so \(a(\theta_*)=\Lambda_*\).
Its difference from \(\Lambda_*\) is
\(O(\norm{\theta-\theta_*})\); the resulting error divided by
\(\sigma\asymp\norm{\theta-\theta_*}\) is absorbed in the
\(\norm{f}\) term. This proves \eqref{eq:tree-cost}.
\end{proof}

\begin{table}[htbp]
\centering\small
\caption{The complete local fixed-boundary geometry.
Each listed scale is comparable to the Schur residual \(\sigma\);
the 25-vertex expression uses the analytic coordinates in
\eqref{eq:25-normal-form}.}
\label{tab:bad-loci}
\begin{tabular}{@{}rrll@{}}
\toprule
\(V\)&ambient dimension&local bad locus&residual scale\\\midrule
17&26&smooth, dimension 4&\(\dist(\theta,\Sigma_{17})\)\\
25&38&dimensions 6 and 5; intersection 4&
\(\norm{q}+|b|(|c|+|\varphi|)\)\\
31&26&one transverse isolated point&\(\norm{\theta-\theta_*}\)\\
41&34&one transverse isolated point&\(\norm{\theta-\theta_*}\)\\
\bottomrule
\end{tabular}
\end{table}

\begin{figure}[htbp]
\centering\includegraphics[width=\linewidth]{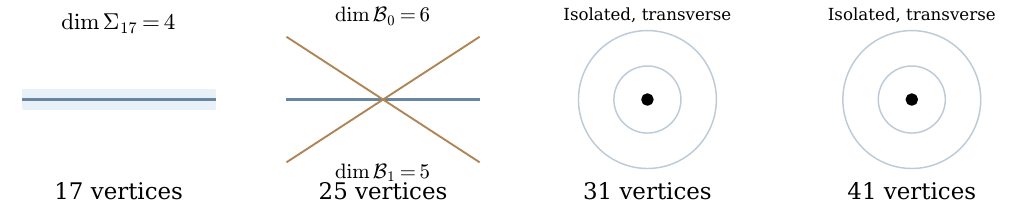}
\caption{Schematic comparison of the four fixed-boundary bad loci.
The dimensions refer to the full internal-coordinate spaces,
respectively 26, 38, 26 and 34. The sketch for 25 vertices indicates
intersecting components, not a claimed topological normal form of its
singular component.}
\label{fig:badsets}
\end{figure}
The construction parameters in the rational families move their
boundary vertices. They do not contradict fixed-boundary isolation.
Likewise, local isolation by itself would not prove the linear distance
scale: the transverse first-order certificates above are needed.
 \section{Arbitrary defect on interior trees}\label{sec:ports}
Throughout this section the vertex labels refer to the original 41-vertex
mesh \(\Dstar\) of Section~\ref{sec:core}. We retain its boundary-patch
construction before turning to indecomposable examples.
We first construct the local extensions used to join copies of $\Dstar$
in Theorem~\ref{thm:chain}.
\subsection{Local extensions and trace lifting}
For distinct points $U,V$, put
$\ell_{UV}(x,y)=(V_x-U_x)(y-U_y)-(V_y-U_y)(x-U_x)$.
Two cubic pieces match to first order along a line $\ell=0$ exactly when
their difference lies in $\ell^2\PP_1$.
An \emph{ear} adds $ACv$ along a boundary edge $AC$; a \emph{cap} adds
$ABv,BCv$ along a boundary chain $A-B-C$, with $B$ becoming interior and
$BA,BC,Bv$ distinct lines. In both cases $v$ is new, the triangles are
nondegenerate, and the attachment meets the old disk only along that edge or
chain;
Section~\ref{sec:composition} verifies these geometric hypotheses for the attachments used below.

\begin{lemma}\label{lem:extensions}
For either attachment $\Delta\subset\Delta'$, restriction
$\Sspace(\Delta')\to\Sspace(\Delta)$ is surjective, with nullity three for an ear
and two for a cap. With the old spline fixed, one may prescribe
$(s(v),\nabla s(v))$ for an ear, or $s(v)$ for a cap. In both cases,
$\im E_{\Delta'}=\im E_\Delta\times\R$ and $\delta_V(\Delta')=\delta_V(\Delta)$.
\end{lemma}
\begin{proof}
For an ear, the extensions are $p+\ell_{AC}^2q$, $q\in\PP_1$, where $p$ is the
old adjacent piece. Since $\ell_{AC}(v)\ne0$, multiplication by $\ell_{AC}^2$
is invertible on first-order jets at $v$.

For a cap with old pieces $p_A,p_C$ adjacent to $AB$ and $BC$, (smoothing)
cofactor corrections give
\begin{equation}\label{eq:cap-map}
 \ell_{AB}^2q_A-\ell_{BC}^2q_C-\ell_{Bv}^2q_v=p_C-p_A,\ q_A,q_C,q_v\in\PP_1.
\end{equation}
In coordinates centered at $B$, the right side lies in $\HH_2\oplus\HH_3$,
where $\HH_k$ denotes the homogeneous polynomials of degree $k$ in two
variables.
The squares of the three distinct linear forms form a basis of $\HH_2$; any two
of the spaces $\ell_i^2\HH_1$ meet trivially by unique factorization, and each
pair spans $\HH_3$.
Thus the map in \eqref{eq:cap-map} is surjective onto $\HH_2\oplus\HH_3$, of
rank $3+4=7$.
For fixed $p_A,p_C$, the solution triples of \eqref{eq:cap-map} form a nonempty
affine space of
dimension $9-7=2$. The new pieces determine all cofactors, so restriction also
has nullity two.

In the kernel, the quadratic part forces the constant terms of all cofactors
to vanish.
Use affine coordinates $(t,u)$ with $B=(0,0)$, $v=(1,0)$ and $Bv=\{u=0\}$.
Each piece has the form $(\lambda t+\mu u)^2(at+bu)$, whose trace coefficients
are $(\lambda^2a,\lambda^2b+2\lambda\mu a)$.
Since $\lambda\ne0$, the common trace $(\alpha t^3,\beta t^2)$ gives an injective
map from the two-dimensional kernel to $\R^2$, hence an isomorphism.
In particular, $s(v)=\alpha$ is arbitrary.
\end{proof}

\noindent\begin{minipage}[c]{0.68\linewidth}
Define the \emph{boundary patch} $\Pi$ by
\begin{equation}\label{eq:port-local}
\begin{gathered}
A=(-1,0),\ B=(0,1),\ C=(1,0),\\
W=(5/7,0),\ X=(-44,67),\ Y=(-351/4,133),\\
L=Y-(1/100,0),\ R=Y+(1/100,0).
\end{gathered}
\end{equation}
Attach to $ABW$ and $BCW$ the six triangles, in the order used below,
\begin{equation}\label{eq:port-faces}
 ABX,\ BCX,\ AXY,\ XCY,\ AYL,\ YCR.
\end{equation}
\end{minipage}\hfill
\begin{minipage}[c]{0.29\linewidth}
\centering
\resizebox{0.90\linewidth}{!}{%
\begin{tikzpicture}[x=1cm,y=1cm,line width=0.3pt,
  every node/.style={font=\small,inner sep=1.5pt}]
\coordinate (A) at (-1.4,0);
\coordinate (C) at (1.4,0);
\coordinate (W) at (0.8,0);
\coordinate (B) at (0,0.55);
\coordinate (X) at (-0.25,1.15);
\coordinate (Y) at (-0.4,1.9);
\coordinate (L) at (-1.05,1.9);
\coordinate (R) at (0.25,1.9);
\fill[gray!10] (A)--(B)--(C)--cycle;
\foreach \a/\b/\c in {A/B/X,B/C/X,A/X/Y,X/C/Y,A/Y/L,Y/C/R}
  \fill[blue!7] (\a)--(\b)--(\c)--cycle;
\foreach \a/\b in {A/W,W/C,A/B,B/C,B/W,A/X,B/X,C/X,A/Y,X/Y,C/Y,A/L,Y/L,Y/R,C/R}
  \draw (\a)--(\b);
\foreach \v in {A,B,C,W,X,Y,L,R}
  \fill (\v) circle (0.8pt);
\node[below left] at (A) {$A$};
\node[below right] at (C) {$C$};
\node[below] at (W) {$W$};
\node[left] at (B) {$B$};
\node[right] at (X) {$X$};
\node[above] at (Y) {$Y$};
\node[above left] at (L) {$L$};
\node[above right] at (R) {$R$};
\end{tikzpicture}%
}
 \captionsetup{hypcap=false}
\captionof{figure}{The reference boundary patch $\Pi$ for trace lifting (not to scale).}\label{fig:port-patch}
\end{minipage}\par\smallskip
On $\Gamma=LY\cup YR$, write
$\Tr_\Gamma s=((s|_e,\partial_y s|_e))_{e=LY,YR}$.
Let $J(\Gamma)$ be the space of pairs of a cubic value trace
and a quadratic transverse trace on each edge of $\Gamma$,
with matching first-order jets at $Y$, and $J_0(\Gamma)$ its subspace with zero
values at $L,Y,R$. Then $\dim J=14-3=11$ and $\dim J_0=8$.

Extending $p_0=y^2(z_0+z_1x+z_2y)$ on $ABW,BCW$ by two cap extensions and two
ear extensions defines a space $\mathcal P$ of dimension $3+2+2+3+3=13$.
Its functions extend by zero across $AW\cup WC\subset AC$ and have zero
first-order jets at $A,C$. Let $\mathcal P_0$ consist of those vanishing at every vertex of $\Pi$.

\begin{proposition}\label{prop:port}
The map
$\mathcal T:\mathcal P\to J(\Gamma)\oplus\R^2$,
$s\mapsto(\Tr_\Gamma s,s(B),s(X))$, is an isomorphism, as is
$\Tr_\Gamma:\mathcal P_0\to J_0(\Gamma)$.
Both inverse maps (the trace liftings) are supported on $\Pi$ after zero
extension.
\end{proposition}
\begin{proof}
\textit{Source coordinates.}
For a cap with center $P$, new vertex $v$, and old pieces $p_A,p_C$, write its
$C^1$-preserving corrections as
\begin{equation}\label{eq:cap-params}
\begin{aligned}
 \widehat p_A&=p_A+\ell_{AP}^2(a_0+a_1(x-P_x)+a_2(y-P_y)),\\
 \widehat p_C&=p_C+\ell_{CP}^2(a_3+a_4(x-P_x)+a_5(y-P_y)).
\end{aligned}
\end{equation}
Set $d=v-P$, $n=(-d_y,d_x)$.
The first-order jets of the old pieces at $P$ agree, so matching across $Pv$
requires precisely that
the $t^2,t^3,tu,t^2u$ coefficients of
$(\widehat p_A-\widehat p_C)(P+td+un)$ vanish.
For fixed $a_4,a_5$, let $K(P,v)$ be the coefficient matrix with rows in the
order $t^2,t^3,tu,t^2u$ and columns indexed by $(a_0,a_1,a_2,a_3)$.
Write $\ell_{AP}(P+td+un)=\alpha t+\beta u$ and
$\ell_{CP}(P+td+un)=\gamma t+\delta u$.
Grouping the quadratic and cubic rows gives
$\det K(P,v)=2\alpha^5\gamma(\alpha\delta-\beta\gamma)\|d\|^2\ne0$:
the distinct lines $AP,CP,Pv$ give $\alpha\gamma\ne0$, and the independence of
$\ell_{AP}$ and $\ell_{CP}$, expressed in the coordinates $(t,u)$, gives
$\alpha\delta-\beta\gamma\ne0$.
The two cap problems therefore have unique solutions:
first $(P,v)=(B,X)$, $(p_A,p_C)=(p_0,p_0)$, $(a_4,a_5)=(z_3,z_4)$ gives $p_1,p_2$;
then $(P,v)=(X,Y)$, $(p_A,p_C)=(p_1,p_2)$, $(a_4,a_5)=(z_5,z_6)$ gives $p_3,p_4$.
The ear pieces are
\[
 p_5=p_3+\ell_{AY}^2(z_7+z_8x+z_9y),\ 
 p_6=p_4+\ell_{CY}^2(z_{10}+z_{11}x+z_{12}y).
\]
Use $p_0$ on $ABW,BCW$ and $p_1,\ldots,p_6$ on the six triangles of
\eqref{eq:port-faces}, respectively.
Restriction in turn to the original triangles, the cap pieces and the ear
pieces recovers $z_0,\ldots,z_{12}$ in that order, proving that they are
coordinates on $\mathcal P$.

\textit{Target coordinates.}
Use the ordered list
\[
 \bigl(j^1s(L),j^1s(Y),j^1s(R),
 s_y\bigl(\tfrac{L+Y}{2}\bigr),s_y\bigl(\tfrac{Y+R}{2}\bigr),s(B),s(X)\bigr).
\]
The first 11 entries are unisolvent for $J(\Gamma)$: endpoint values and tangential
derivatives determine each cubic value trace, while endpoint and midpoint
transverse derivatives determine each quadratic transverse trace.
Evaluate with $p_5$ on $LY$, $p_6$ on $YR$, $p_0$ at $B$, and $p_1$ at $X$.
The coefficients of $z_0,\ldots,z_{12}$ specify the matrix $M_{\mathcal P}$ of $\mathcal T$.
Using these orderings, we clear the denominators of each row and apply
fraction-free Gaussian
elimination. Rescaling the rows gives
\[
 \det M_{\mathcal P}
 =\frac{3^2\,7^{13}\,11^3\,19^{13}\,31^4\,67}{2^{36}\,5^{25}}>0.
\]
Thus $\mathcal T$ is an isomorphism.
For $\tau\in J(\Gamma)$, let $s=\mathcal T^{-1}(\tau,0,0)$.
Its values at $A,C,W$ vanish by construction, and those at $B,X$ by the target conditions.
All vertex values vanish exactly when $\tau\in J_0$, proving the second
assertion.
\end{proof}
Affine pullback preserves matching, vertex vanishing and support.
We transport the transverse derivative and compare the shared traces in a
single transverse direction common to both sides (not necessarily the edge
normal).

\subsection{The enlarged triangulation and its repeated gluing}\label{sec:composition}
Add to $\Dstar$, in sector zero,
\begin{equation}\label{eq:port-global}
\begin{aligned}
 X_0&=(7/8,3/8),&Y_0&=(239/192,323/576),\\
 L_0&=(749/600,1007/1800),&R_0&=(993/800,1349/2400),
\end{aligned}
\end{equation}
and their negatives in sector four.
The triangles \eqref{eq:port-faces}, with every vertex subscript set in turn to
$0$ and to $4$, define the enlarged triangulation $\Dhat$ in
Figure~\ref{fig:expanded}.
Its \emph{interfaces} are $\Gamma_+=L_0Y_0\cup Y_0R_0$ and $\Gamma_-=-\Gamma_+$.
The~affine~map
\begin{equation}\label{eq:affine-port}
 \mathcal A(x,y)=\tfrac12(A_0+C_0)+\tfrac x2(C_0-A_0)
               +y\bigl(B_0-\tfrac12(A_0+C_0)\bigr)
\end{equation}
takes $\Pi$ to sector zero; the determinant of its linear part is $-1/576$.
Central reflection gives the sector-four patch.
With $Y=Y_0$, set
\begin{equation}\label{eq:chain}
 \Dhat_j=\Dhat+2(j-1)Y,\ \Delta_m=\bigcup_{j=1}^m\Dhat_j,
\end{equation}
and let $\Lambda_j$ be the translated functional \eqref{eq:Lambda},
extended by zero outside its copy of $\Dstar$; we call each translate $\Dhat_j$
a \emph{block} and its copy of $\Dstar$ the \emph{core}.

\begin{figure}[tb]
\centering
\includegraphics[width=\linewidth,trim=0 20bp 0 0,clip]{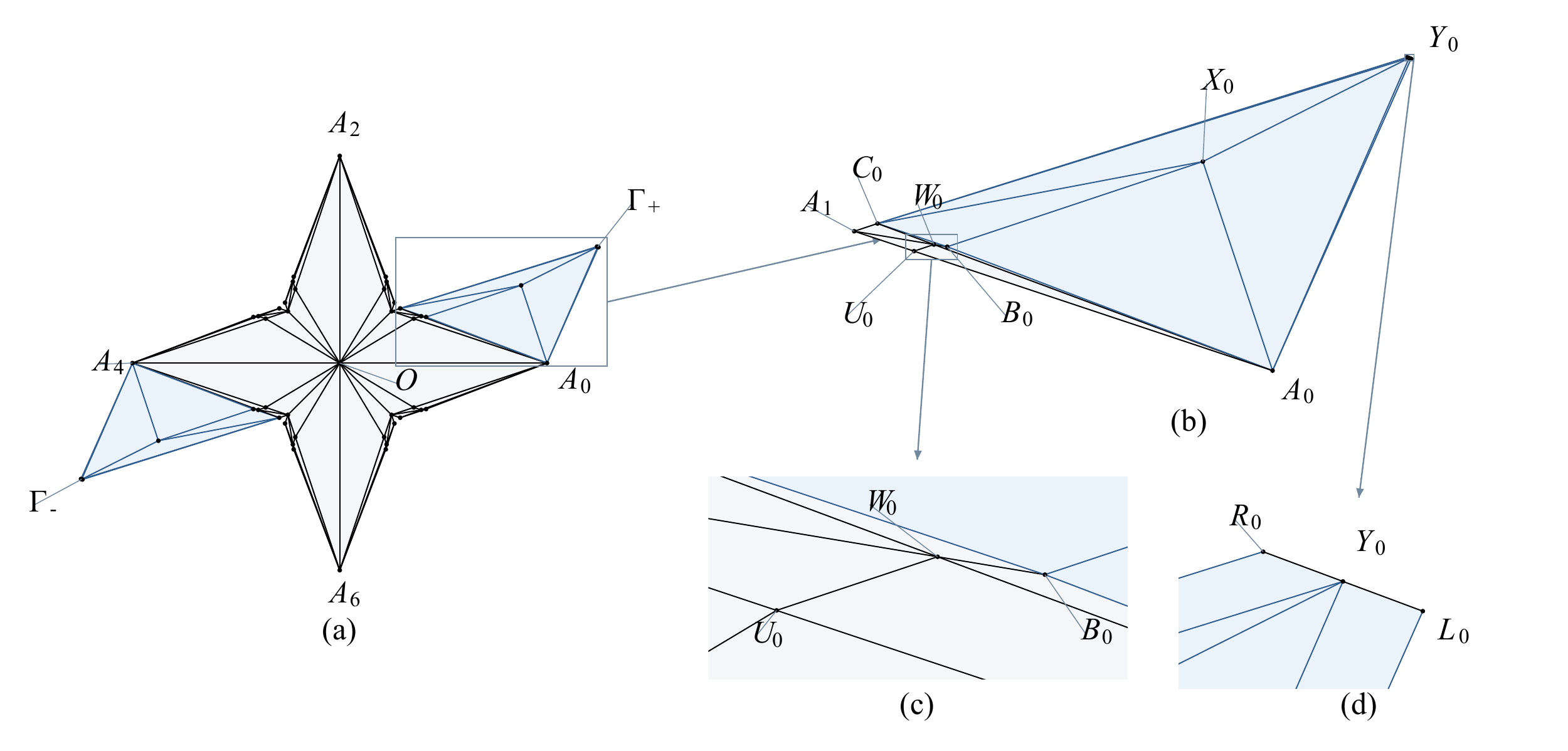}
 \caption{(a) The triangulation $\Dhat$; (b) its sector-zero patch; (c) detail near $U_0,W_0,B_0$;
(d) detail of $\Gamma_+$. Shaded triangles are added to $\Dstar$.}\label{fig:expanded}
\end{figure}
\begin{samepage}
\begin{theorem}\label{thm:chain}
For every $m\ge1$, $\Delta_m$ is a conforming nondegenerate triangulated disk
without singular interior vertices. Its interior graph $G_I(\Delta_m)$
is a tree of maximum degree eight. The following identities hold:
\begin{equation}\label{eq:chain-results}
\begin{gathered}
 V_B=24m+4,\ V_I=22m-1,\ V=46m+3,\ N=68m,\\
 \dim\Sspace(\Delta_m)=116m+11=3V_B+2V_I+1,\\
 \im E_{\Delta_m}=\bigcap_{j=1}^m\ker\Lambda_j,\ 
 \delta_V(\Delta_m)=m,\ \rank E_{\Delta_m}=45m+3.
\end{gathered}
\end{equation}
Its minimum triangle angle has a positive lower bound independent of $m$.
\end{theorem}
\end{samepage}
\begin{proof}
\textit{Step 1: Establish the geometry and counts.}
The strict convex combinations $B=(11A+X+55C)/67$ and
$X=(155A+536Y+373C)/1064$ in \eqref{eq:port-local} give
$\triangle ABC\subset\triangle AXC\subset\triangle AYC$.
The nondegenerate ears $AYL,YCR$ lie outside
$\triangle AYC$ across the edges $AY$ and $CY$, extending the boundary to
$LY\cup YR$.
In global coordinates the additions lie in $x\ge y\ge0$, on the side of
$A_0C_0$ opposite the old sector; they meet it only along the attachment arcs
and avoid the sector-four patch. Thus $\Dhat$ is a conforming disk.

Each patch adds four vertices and six triangles and attaches the path
$W_i-B_i-X_i$ to the interior tree, giving $(V,V_B,V_I,N)=(49,28,21,68)$.
The directions to $A_i,W_i,C_i$ at $B_i$, and to $A_i,B_i,C_i$ at $X_i$,
are distinct; the stars of the old interior vertices are unchanged. Hence
$\sigma=0$.

Put $\nu(x,y)=19x+51y$, $h=209/4$.
At the old vertices, $|\nu|\le51<h$, while $\nu(X_0)=143/4<h$ and
$\nu(L_0)=\nu(Y_0)=\nu(R_0)=h$.
Since $\nu$ is affine, its extrema on each triangle occur at vertices; this and
central reflection give $|\Dhat|\subset\{-h\le\nu\le h\}$
and $|\Dhat|\cap\{\nu=\pm h\}=|\Gamma_\pm|$.
Only consecutive translates of this slab meet, each pair in a single line.
Since $L_0+R_0=2Y$, translation
by $2Y$ carries $(-R_0,-Y,-L_0)$ to $(L_0,Y,R_0)$, so consecutive blocks share
exactly the two-edge interface and their union is a disk.
Each gluing identifies three vertex pairs, makes the identified vertex $Y$
interior, and reduces the boundary-vertex count by four. The path $X-Y-X'$
joins the block trees without creating a cycle; its vertices have degree two
in $G_I$, while $Y$ has eight incident edges and is not singular. The stated counts and
graph properties follow.

\textit{Step 2: Match the traces and compute the dimension.}
Lemma~\ref{lem:extensions} gives
$\dim\Sspace(\Dhat)=107+2(2\cdot2+2\cdot3)=127$ and
$\im E_{\Dhat}=\ker\Lambda$.
Proposition~\ref{prop:port} prescribes arbitrary traces on each interface while
keeping zero trace on the other; lifts of $J_0$ data also vanish at every
vertex.

Let $\Gamma_j=|\Dhat_j|\cap|\Dhat_{j+1}|$ and
$D(s_1,\ldots,s_m)=(\Tr_{\Gamma_j}s_j-\Tr_{\Gamma_j}s_{j+1})_{j=1}^{m-1}$,
using a common transverse direction on each interface.
Its kernel is the space of matching tuples, i.e., of global $C^1$ splines. To
prescribe component $j$ of $D$,
use the trace lifting on the $\Gamma_+$-patch of block $j$; this changes no other
component, so
\begin{equation}\label{eq:gluing}
 0\longrightarrow\Sspace(\Delta_m)\longrightarrow
 \bigoplus_{j=1}^m\Sspace(\Dhat_j)
 \xrightarrow{\,D\,}\bigoplus_{j=1}^{m-1}J(\Gamma_j)\longrightarrow0
\end{equation}
is exact. Hence $\dim\Sspace(\Delta_m)=127m-11(m-1)=116m+11$.

\textit{Step 3: Determine the vertex image.}
Restriction to each core forces $\Lambda_j(f)=0$.
Conversely, block interpolants for such data have jumps in $J_0(\Gamma_j)$,
canceled independently by the vertex-vanishing lifts of Step 2. This proves
the formula for $\im E_{\Delta_m}$. Only $\Lambda_j$ has a nonzero coefficient
at $O_j$, proving independence and hence defect $m$ and rank $V-m$.
The 68 fixed nondegenerate triangle shapes of $\Dhat$ give the uniform angle bound.
\end{proof}

\par\smallskip\noindent
\begin{minipage}{\linewidth}
\centering
\captionsetup{hypcap=false}
\includegraphics[width=\linewidth]{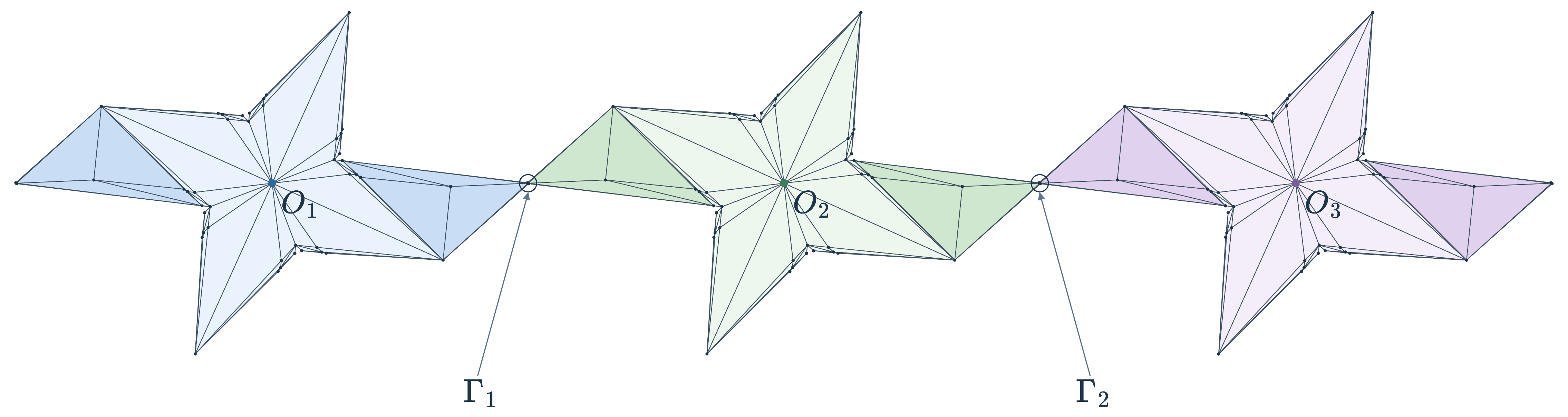}
\captionof{figure}{Three translated blocks forming $\Delta_3$, rotated for display.
Circles mark the short two-edge interfaces $\Gamma_1$ and $\Gamma_2$.}
\label{fig:chain-three}
\end{minipage}
\par\smallskip

Consequently $\dim\ker E_{\Delta_m}=71m+8$.
As $\Lambda_1,\ldots,\Lambda_m$ form a basis of the space of vertex relations
and $O_j$ occurs only in $\Lambda_j$, deleting the $m$ evaluations at the
centers $O_j$ leaves an independent family, and since $\delta_V=m$, no smaller
set of deletions suffices.
 \section{Obstructions created by complete interfaces}\label{sec:interfaces}
\subsection{Two notions of indecomposability}
The chains of Section~\ref{sec:ports} have independent, disjoint old
cores. We now distinguish two ways of excluding that description.
A relation space is \emph{support-indecomposable} if it is not a direct
sum of two nonzero relation spaces on disjoint vertex subsets.

\begin{definition}\label{def:strong}
A triangulated disk is \emph{strongly indecomposable} if it has positive
vertex defect and does not contain two connected normal bad submeshes
with disjoint triangle sets.
A submesh is a union of original triangles forming a connected polygonal
region, possibly with holes, with each vertex star a path or a circle.
Its spline space is its own full \(S_3^1\), without boundary conditions.
\end{definition}
These are different properties. A vertex relation on any bad submesh,
extended by zero, is a relation on the full mesh. However a zero
coefficient does not remove the geometric role of a triangle or vertex.

An intermediate rational 83-vertex construction illustrates
support-indecomposability.
Attach the ear \((A_0,X,B_0)\), \(X=(1,1/100)\), to the 41-vertex
mesh. Reflect the resulting 42-vertex block by
\((x,y)\mapsto(2-x,y)\) and glue along \(VX\), \(V=A_0=(1,0)\).
Add \(Y=(1,-1/100)\) and the two cap triangles
\[
 (V,C_7^L,Y),\qquad(V,Y,C_7^R).
\]
The resulting mesh has \((V,V_B,V_I,N,E_I)=(83,48,35,116,150)\),
dimension 215, and precisely the two old relations
\(\Lambda_L,\Lambda_R\). The five zero-value complete edge traces on
\(VX\) have a full lifting from the ear-enlarged block, so those two
conditions are sufficient as well as necessary.
The relation columns at \(O_L,O_R,V\) have three different projective
directions:
\[
 (-756,0)^{\mathsf T},\quad(0,-756)^{\mathsf T},\quad(11,11)^{\mathsf T}.
\]
A basis with disjoint supports would allow only two such directions.
Thus the relation space is support-indecomposable, and two
vertex-disjoint bad submeshes are impossible.
This argument alone does not prove \cref{def:strong}; the following
constructions prove the stronger triangle-set statement.

\subsection{A certified interpolatory block with one gradient requirement}
There is an exact 17-vertex \emph{good} block, with the same face list
and diamond boundary as Section~\ref{sec:example17}, such that
\begin{equation}\label{eq:good-gradient}
 (f,g_{P_0})\text{ is realizable}\quad\Longleftrightarrow\quad
 \lambda^{\mathsf T}f+\gamma^{\mathsf T}g_{P_0}=0,
 \qquad \lambda_O=1,\quad\gamma\ne0.
\end{equation}
Its dimension is 39, its vertex evaluation is surjective, and the
coefficient \(\lambda_{W_0}\ne0\).
It is a different algebraic geometry from the bad 17-vertex example.

Its exact definition is useful. In the 37 normalized joint variables
fix the five rational coordinates specified in the root certificate.
Solve the 32 components \(G_v=0\), \(v\ne P_0\), in the remaining
variables. A rational box of radius \(10^{-60}\) and a rational
contraction matrix certify its unique zero; the derivative norm is
less than \(6.25\,10^{-54}\).
At this zero \(G_{P_0,x}\) is in a strictly positive interval near
\(8.2379806\,10^{-8}\), and \(R_O\ne0\).
The eleven pressure columns are independent. Hence the only
pressure with all other vertex forces zero is the constructed one;
the missing \(P_0\) force prevents a pure vertex relation, while its
weighted equation gives \eqref{eq:good-gradient}.
The same rank calculation proves sufficiency of the mixed condition.
The root certificate specifies all frozen values, box centers,
equations, inverses and strict inequalities.

Under \(p'=v+A(p-P_0)\), the condition becomes
\(\lambda^{\mathsf T}f+(A\gamma)^{\mathsf T}g_v=0\).
This converts the block into a prescribed scalar requirement on a
shared two-component gradient.

\subsection{Four good blocks and a 69-vertex strong example}
Let \(R\) be the quarter-turn matrix and put
\[
 A_j=R^j\begin{pmatrix}-1&0\\-j/50&-1/5\end{pmatrix},
 \qquad p_v^{(j)}=A_j(p_v-P_0),\quad0\le j<4.
\]
All four \(P_0\) vertices become \(V=0\).
The four thin diamonds occupy disjoint angular sectors.
With \(U_j=p_{P_3}^{(j)}\), \(L_j=p_{P_1}^{(j)}\), add
\[
 B_j=U_j+L_{j+1},\qquad
 (B_j,L_{j+1},V),\quad(B_j,V,U_j).
\]
These bridges make a legal disk.

The complete first-order trace on a two-edge noncollinear chain has
dimension eleven: the two seven-dimensional edge traces share a
three-dimensional vertex jet. Two cubic bridge pieces have dimension
thirteen after matching their common edge. Their map to those eleven
trace coordinates and the new tip value is surjective, with a
one-dimensional kernel. This follows either from the squared-line
cofactor equations or the exact physical coefficient lifting supplied
with the bridge certificate.
The unrestricted extension, before prescribing the new value,
has a two-dimensional kernel.
Thus the bridges match entire edge traces, not just their endpoint data.

\begin{theorem}[Gradient-coupled strong example]\label{thm:69}
The resulting mesh has
\[
 (V,V_B,V_I,N,E_I)=(69,16,53,120,172),\qquad
 \dim S_3^1=155,\quad\rank E=67,\quad\dd=2.
\]
Each of its four 17-vertex blocks is interpolatory.
The whole mesh is strongly indecomposable.
\end{theorem}
\begin{proof}
Let \(\Gamma\) have rows \((A_j\gamma)^{\mathsf T}\), and put
\(\ell_j(f)=-\lambda^{\mathsf T}f^{(j)}\).
Its rows are pairwise independent. To see this, write the original
nonzero force as \((a,b)\), \(a\ne0\).
For opposite rows the determinants have a nonzero factor
\(a^2(h_i-h_j)\), \(h_j=j/50\).
For neighboring rows, up to sign, they are
\[
 a^2+(h_i a+b/5)(h_j a+b/5)
 \ge a^2\left(1-\frac{(h_i-h_j)^2}{4}\right)>0.
\]
The complete mixed block image and bridge lifting imply
\begin{equation}\label{eq:69-image}
 f\text{ is realizable}\quad\Longleftrightarrow\quad
 \Gamma g_V=\ell(f)\text{ for some }g_V\in\R^2.
\end{equation}
Each private \(O_j\) occurs only in \(\ell_j\), with coefficient \(-1\).
Thus \(\ell\) is onto \(\R^4\), and a basis \(T\) of
\(\ker\Gamma^{\mathsf T}\) gives exactly two relations
\(T^{\mathsf T}\ell(f)=0\).
The full function-space count is
\(4\cdot39-9+4\cdot2=155\):
sharing a vertex jet removes nine dimensions, and each bridge adds two.

Every nonzero vector in \(\ker\Gamma^{\mathsf T}\) has at least three
nonzero entries. Its relation therefore loads at least three of the
four private five-degree vertices \(W_0^{(j)}\).
The boundary-cardinal lemma proved in
\cref{lem:boundary-cardinal} shows that a relation cannot load a boundary
vertex of complete degree at most four.
Consequently a normal bad submesh loading an original five-degree
interior vertex must contain at least four of its five incident triangles.
Two bad submeshes each load at least three of the four specified vertices,
so they load a common one. Each would occupy at least four of its five
triangles, which precludes disjoint triangle sets.
The argument uses local normal stars and allows holes.
\end{proof}

\begin{figure}[htbp]
\centering\includegraphics[width=\linewidth]{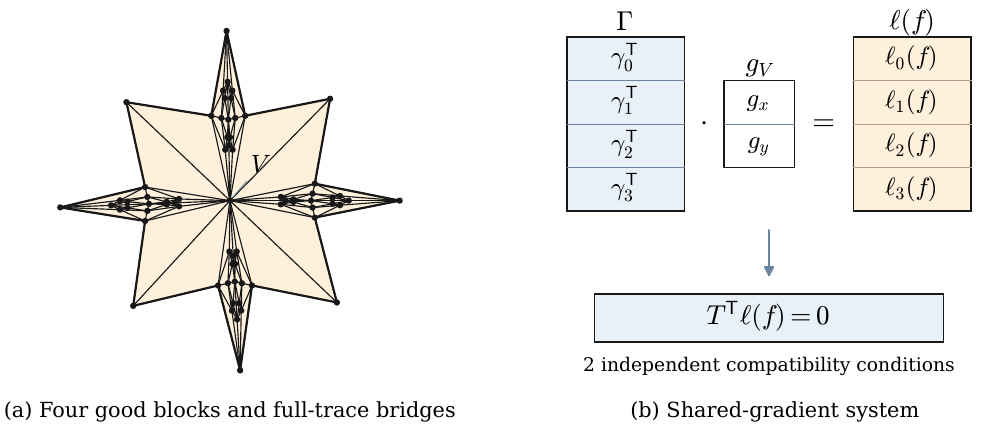}
\caption{The actual 69-vertex mesh and its shared-gradient equations.
The four components are individually interpolatory.
Panel (b) displays the four-row system \(\Gamma g_V=\ell(f)\);
the columns of \(T\) span \(\ker\Gamma^{\mathsf T}\), giving the two
independent compatibility conditions. Geometric coordinates are displayed at the certified
interval midpoints; the mixed block is defined by its exact root.}
\label{fig:69}
\end{figure}

This model also describes the complete fixed-boundary local geometry.
There are analytic coordinates
\[
 (\zeta,\rho_0,\rho_1,\rho_2,\rho_3)
 \in\R^{22}\times(\R^{21})^4
\]
in the full 106-dimensional internal-coordinate space.
The normal blocks \(\rho_j\) are independent; \(\rho_j=0\) means that
block \(j\) retains its scalar gradient requirement.
The private-node derivative has certified rank 21.
If \(k\) blocks retain that requirement, then
\[
 \dd=\max(k-2,0),\qquad
 \Sigma_{69}=\bigcup_{|J|=3}\{\rho_j=0:\ j\in J\}.
\]
There are four smooth 43-dimensional components, with common
22-dimensional intersection; no additional higher-order conditions
remain after the complete interface reduction.
Writing the four normal norms in decreasing order, one obtains
\(\dist(\theta,\Sigma_{69})\asymp\norm{\rho}_{(2)}\).
The smallest residual singular value has the same scale, since any two
columns of a left-kernel basis of \(\Gamma\) are independent.
Hence the worst unit-data interpolation cost is comparable to the
reciprocal Euclidean distance to this nonsmooth bad set.
Moving the common interior vertex \(V\) by any sufficiently small
nonzero vector removes both constraints; a private-node move affects
only one block and leaves at least one.

The same construction with \(q\ge2\) suitably thin rationally placed
good blocks gives
\[
 V=17q+1,\quad V_I=13q+1,\quad V_B=4q,\quad N=30q,\quad
 \dd=q-2.
\]
For \(q\ge3\) its local bad set is
\(\bigcup_{|J|=3}\{\rho_j=0:j\in J\}\), in
\(\R^{5q+2}\times(\R^{21})^q\).
For large \(q\), this model alone does not prove strong
indecomposability. Pairwise-crossing loops will supply that property
for arbitrarily large defects.

\subsection{The exact capacity of one complete edge}
Let two interpolatory disks meet only along a full edge \(e=[a,b]\),
with no extra vertex in its interior. Use one common transverse direction.
An edge first-order trace has dimension seven: a cubic value trace and a
quadratic transverse-derivative trace.
At zero endpoint values its trace space \(J_e\) has dimension five.
For \(S_i=S_3^1(\Delta_i)\), set
\[
 K_i=\ker E_i,\qquad W_i=\Tr_e K_i\subset J_e.
\]
These are images of the entire zero-value spline spaces.

\begin{theorem}[Complete-edge capacity]\label{thm:edge-capacity}
If the union is a conforming disk, then
\begin{equation}\label{eq:edge-capacity}
 \dd(\Delta_1\cup\Delta_2)=5-\dim(W_1+W_2).
\end{equation}
Every defective such union is strongly indecomposable.
The upper bound five is attained.
\end{theorem}
\begin{proof}
The full trace-difference map \(S_1\oplus S_2\to\R^7\) is onto:
a global cubic on one side already realizes every single-edge trace.
Thus the glued space has dimension \(\dim S_1+\dim S_2-7\).
The zero-value glued space has dimension
\(\dim K_1+\dim K_2-\dim(W_1+W_2)\).
Subtracting these dimensions and using
\(V=V_1+V_2-2\) proves \eqref{eq:edge-capacity}.
Equivalently, a data-dependent initial trace jump is correctable exactly
when its class in \(J_e/(W_1+W_2)\) is zero.

A submesh of an interpolatory mesh is interpolatory: a relation on the
submesh, extended by zero, would contradict interpolation on the whole.
Every connected normal bad submesh of the union must therefore cross
from one side to the other. Its triangle dual graph is connected, and
the only dual channel between the two sides is \(e\).
It must contain the same two triangles incident to \(e\).
Two such bad submeshes cannot be triangle-disjoint.

For sharpness, the constructions below give a good block with \(W_e=0\).
Its mirror has the same zero trace space, so the defect of their union
is five.
\end{proof}

Here are the actual constructions reaching the last assertion.
Four gradient-good 17-point blocks, two attached at each endpoint of a
central triangle, fix both endpoint gradients from their vertex data.
Use the reference triangle \(a=(0,0),b=(1,0),c=(1/2,1)\).
Place the blocks using
\[
 A(d)=\tfrac1{100}
 \begin{pmatrix}-d_x&d_y/10\\-d_y&-d_x/10\end{pmatrix},\qquad
 d=(-1,2),(-2,1),(2,1),(1,2),
\]
at \(a,a,b,b\), and fill the four gaps with the two-triangle bridges.
This 71-vertex good block has
\[
 W_e=\operatorname{span}\{(0,0,0,0,1)\}.
\]
The nonzero mode is realized by the central cubic bubble
\(4(1-x-y/2)(x-y/2)y\), extended through the bridges with zero
new values. Mirroring gives a 140-vertex, defect-four strong example.

To remove the last edge mode, replace the central triangle by an
\emph{edge-good} 17-vertex block, affinely mapped to the unit square.
Its complete mixed image has exactly one condition
\begin{equation}\label{eq:edge-good}
 L(f)+\beta_a\cdot g_a+\beta_b\cdot g_b+
 \eta n_{ab}=0,\qquad \eta=1/300000,
\end{equation}
where \(n_{ab}\) is the midpoint derivative in direction \((0,1)\).
All its values and both endpoint gradients can be prescribed arbitrarily;
the nonzero \(\eta\) determines the remaining edge mode.
This is another separately certified algebraic good block.
In its defining pressure system the face sum has source \(10^{-5}\)
on the triangle incident to the distinguished edge, zero on other
triangles, and all forces away from its two endpoints vanish.
Seven frozen rational coordinates and thirty balance equations define
the certified root. The root and the closed-boundary image certificate
are included with the evidence.

With the same four auxiliary gradient blocks and bridges it gives an
85-vertex interpolatory block with \(W_e=0\). The two-sided mirror
therefore has
\[
 (V,V_I,V_B,N)=(168,130,38,296),\quad
 \dim S_3^1=375,\quad\rank E=163,\quad\dd=5.
\]
The five conditions match the two endpoint gradients and midpoint
normal derivative. Private center values prove their independence.
The complete-edge construction also proves strong indecomposability.

\subsection{A central triangle and eight constraints}
Attach three interpolatory arms along the three full edges of a
central triangle \(T\), with no other overlap.
Let \(W_i\subset J_i\), \(\dim J_i=5\), be their zero-value edge images.
The cubics on \(T\) with zero values at its corners have dimension seven,
and their full three-edge trace is injective; denote its image by \(U_T\).
The same data-image quotient gives
\begin{equation}\label{eq:triangle-capacity}
 \dd=15-\dim\bigl((W_1\oplus W_2\oplus W_3)+U_T\bigr)\le8.
\end{equation}
Using three copies of the 85-vertex block with zero edge image reaches
eight. A concrete nonoverlapping placement, for
\(v_0=(0,0),v_1=(1,0),v_2=(0,1)\), is
\[
 p\longmapsto v_i+x(v_{i+1}-v_i)
 +y(10v_i+10v_{i+1}-20v_{i+2}),\qquad i\bmod3.
\]
Each arm has exactly one negative central barycentric coordinate and
the other two positive, which separates the arms.
The resulting mesh has
\[
 (V,V_I,V_B,N)=(252,195,57,445),\quad
 \dim S_3^1=562,\quad\rank E=244.
\]
Every normal bad submesh contains the central triangle:
without it the connected dual graph lies in one good arm.
Thus it is strongly indecomposable even for holed bad submeshes.

The capacities five and eight apply to these specified interface
classes. They do not bound all strongly indecomposable defects.
For example, enlarging the central triangle to a triangulated polygon
can create two disjoint bad ear regions. A different construction is
needed to make the defect unbounded while excluding that separation.

\subsection{Interface dependencies and local bad sets}
The preceding constructions have a common exact local model
\begin{equation}\label{eq:interface-normal-model}
 m_j(\theta)h+\rho_j(\theta)z_j=\ell_j(\theta;f),
 \qquad j=1,\ldots,s,
\end{equation}
where \(h\) contains complete shared trace variables.
Assume the full stable lifting has been proved, \(\ell\) is onto,
the \(\rho_j\) are independent analytic normal blocks, and the row
dependence type of \(M=(m_j)\) is fixed in the neighborhood.
For \(Z=\{j:\rho_j=0\}\), exact elimination yields
\[
 \dd=|Z|-\rank M_Z,\qquad
 \Sigma=\bigcup_{C\ {\rm minimal\ dependent}}
          \{\rho_j=0:\ j\in C\}.
\]
The component codimension is the sum of its normal-block dimensions.
For 140 vertices this gives eight components of dimension 145 in the
208-dimensional geometry space.
For 168 vertices there are eight components of dimension 197 and
36 of dimension 140 in dimension 260; their common intersection has
dimension 56 and defect five.
The first eight use three gradient requirements at one endpoint.
The other 36 use both central edge requirements and two auxiliary
requirements at each endpoint.

Let \(C M=0\) be a full left-kernel basis, with columns \(c_j\).
The residual Gram matrix is
\[
 SS^{\mathsf T}=\sum_j\norm{\rho_j}^2c_jc_j^{\mathsf T}.
\]
If all nonzero row-basis minors stay uniformly separated from zero,
finite-dimensional basis selection shows
\[
 \sigma_{\min}(S)\asymp
 \min_{A\ {\rm minimal\ dependent}}\max_{j\in A}\norm{\rho_j}
 \asymp\dist(\theta,\Sigma).
\]
For the upper bound annihilate all columns outside a minimizing
dependency. For the lower bound, the columns with norm above half
that threshold must contain a uniformly nonsingular basis; otherwise
their complement contains a smaller-threshold row dependency.
This proves inverse-distance worst-data estimates for these actual
interface models. The hypotheses fail for the singular product
normal form of the original 25-vertex mesh, so this argument does
not change its different distance behavior.
 \section{Pairwise-crossing loops and unbounded strong defect}
\label{sec:crossing}
We now prove \cref{thm:main-unbounded}. The construction uses
interpolatory local blocks with complete boundary control.
Its obstruction comes from exact scalar monodromy around loops.
Pairwise crossings prevent two normal bad submeshes from carrying
their obstructions on disjoint triangles.

\subsection{Complete-boundary building blocks}
The gradient-good and edge-good 17-vertex blocks of
Section~\ref{sec:interfaces} have the following stronger property.
Their entire value data together with their entire closed-boundary
first-order trace are realizable if and only if the single respective
mixed equation \eqref{eq:good-gradient} or \eqref{eq:edge-good} holds.
After fixing all values and the complete boundary trace the kernel has
dimension eleven.

For the diamond, after its four corner values have been fixed,
the boundary trace has twelve coordinates: eight corner gradient
components and four edge midpoint derivatives.
The exact certificates give
\[
 \rank\begin{pmatrix}H_0\\T_\partial\end{pmatrix}=51,\qquad
 \rank H_0=40
\]
for either good block, with rational interval nonsingularity bounds.
Hence the zero-value boundary image has dimension eleven.
The known nonzero mixed condition supplies its only restriction.
This proves the full image statement, including all boundary derivatives.

The fillers have an equally complete extension property.
In the triangle \(a=(0,0),b=(1,0),c=(0,1)\), take
\[
 v=(1/3,1/3),\quad z_0=(1/4,3/20),\quad
 z_1=(4/9,4/9),\quad z_2=(1/9,4/9).
\]
First split the triangle at \(v\), and split its three triangles at
\(z_0,z_1,z_2\), respectively. This seven-vertex, nine-triangle
mesh has spline dimension eighteen.
\begin{lemma}[Seven-point extension patch]\label{lem:macro7}
On this fixed mesh one may independently prescribe the full
twelve-dimensional boundary first-order trace and the four interior
vertex values. The remaining kernel has dimension two.
\end{lemma}
\begin{proof}
The complete BB assembly with the stated rational coordinates has
an explicit right inverse for those sixteen data.
The exact coefficient certificate contains a \(30\times16\)
right inverse in reduced BB coordinates; its defining
\(28\times28\) nonsingularity determinant is \(-551124000000\).
Alternatively, in ninety independent physical cubic coefficients
the across-edge smoothness rank is 72, and adding the sixteen
data rows raises it to 88.
The saved rational right inverse verifies these identities column
by column. Thus this is an exact finite-dimensional certificate
for the fixed rational patch, with nullity \(18-16=2\).
\end{proof}
Replacing every triangle of any legal coarse polygonal mesh by its
affine image of this patch allows all final vertex values and all
true boundary traces to be specified together.
Indeed choose common gradients at coarse interior vertices and common
midpoint derivatives on coarse interior edges, then fill each patch.
The procedure matches complete traces and also works on domains with
holes. It is a tool for constructing a new mesh; it is not an
interpolation assertion on an arbitrary unrefined old mesh.

\subsection{Thin transport, corners and crossings}
Map the edge-good block affinely to the unit square, with distinguished
edge \(a=(0,0)\), \(b=(1,0)\).
Its condition is
\[
 L(f)+\beta_a\cdot g_a+\beta_b\cdot g_b+\eta n_{ab}=0,
 \qquad\eta=1/300000.
\]
For \(\lambda,\mu>0\) use the projective map
\begin{equation}\label{eq:thin-projective}
 F(x,y)=\left(\frac{x+\lambda y/2}{1+\lambda y},
                   \frac{\mu y}{1+\lambda y}\right).
\end{equation}
It fixes the full edge \(ab\), maps the square to a thin trapezoid,
and preserves the full cubic spline space by the weighted pullback
\begin{equation}\label{eq:weighted-pullback}
 \widetilde P(X,Y)=
 \left(1-\frac{\lambda Y}{\mu}\right)^3
 P\left(\frac{X-\lambda Y/(2\mu)}{1-\lambda Y/\mu},
        \frac{Y/\mu}{1-\lambda Y/\mu}\right).
\end{equation}
The expression is a polynomial of total degree at most three.
The same nonzero smooth factor is used on all triangles, so it
preserves the complete \(C^1\) conditions and is invertible.
Values undergo nonzero diagonal scaling.

Let \(t_a,t_b\) and \(v_a,v_b\) denote tangential and normal endpoint
derivatives in the new coordinates, and \(n\) the midpoint normal
derivative. Direct differentiation gives the old trace variables:
\begin{align}
 v_a^{o}&=\mu v_a+\tfrac\lambda2t_a+3\lambda f_a,&
 v_b^{o}&=\mu v_b-\tfrac\lambda2t_b+3\lambda f_b,\label{eq:thin-traces}\\
 n^{o}&=\mu n+\tfrac{3\lambda}2(f_a+f_b)
                     +\tfrac{3\lambda}8(t_a-t_b).\nonumber
\end{align}
Values and tangential endpoint derivatives are unchanged.
This seven-coordinate edge transformation has determinant \(\mu^3\ne0\).
Set
\[
 A(\lambda)=\beta_{a,x}+\tfrac\lambda2\beta_{a,y}
                      +\tfrac{3\lambda}8\eta,\qquad
 B(\lambda)=\beta_{b,x}-\tfrac\lambda2\beta_{b,y}
                      -\tfrac{3\lambda}8\eta.
\]
The transformed mixed relation is
\[
 L_\lambda(f)+A t_a+B t_b+
 \mu\beta_{a,y}v_a+\mu\beta_{b,y}v_b+\mu\eta n=0.
\]
Join this block to its mirror along \(ab\). Normal variables change
sign, so adding their equations gives exactly
\begin{equation}\label{eq:thin-transport}
 L_\lambda^+(f)+L_\lambda^-(f)+2A t_a+2B t_b=0.
\end{equation}
It is sufficient: \(\mu\eta\ne0\) solves the shared midpoint derivative,
and the two complete-boundary image theorems realize every remaining
trace coordinate. The joined 32-vertex block has full dimension 71.
On the entire range \(0\le\lambda\le1/200\), exact interval bounds give
\(A,B>0\).

Take \(\mu=\lambda/200\). The two outer square corners approach the
endpoints with simultaneous small angles; both end-cone angles are
\(2\arctan(1/100)\).
Along a line segment the two endpoint tangential gradient states satisfy
\[
 h_{\rm next}=r h_{\rm prev}+\ell(f),\qquad r=-A/B\ne0.
\]
Reversing the identical actual template replaces \(r\) by \(r^{-1}\).
The same certified algebraic root must be reused in every copy:
independent rounding of coordinates would destroy this exact identity.

At a right-angle turn let the unit tangent directions be \(t_-,t_+\).
Put \(\Gamma=t_--t_+\), \(S=t_-+t_+\), and use a gradient-good block
with linear part
\[
 M_s=\sigma[-\Gamma,\ s\vartheta S],\qquad
 \vartheta=1/100,\quad s\in\{1,-1\},
\]
centered at its distinguished boundary point.
Its state equation is
\[
 h_+=q_s h_-+\ell_s(f),\qquad
 q_s=\frac{G_x-s\vartheta G_y}{G_x+s\vartheta G_y}.
\]
The certified force has \(G_x>0\) and
\(|\vartheta G_y|<G_x/2\). In particular \(q_+q_-=1\);
no small component has been set equal to zero.

At a crossing, four thin transport blocks occupy four disjoint cones.
The horizontal state uses \(g_{C,x}\), the vertical state \(g_{C,y}\).
These are the two independent components of the one shared full
gradient. Fill every empty angular sector with
\cref{lem:macro7}. Thus the two state directions remain independent;
no additional vertex or trace condition is imposed by the fillers.

\subsection{The actual family and its counts}
Fix \(n\ge1\), put \(\epsilon=1/(n+1)\), \(a_i=(i+1)\epsilon\),
and consider the \(n\) square boundaries
\[
 Q_i=\partial([a_i,a_i+1]\times[a_i,a_i+1]),\qquad0\le i<n.
\]
Every pair intersects at exactly
\((1+a_i,a_j)\), \((a_j,1+a_i)\), and no three meet.
Each loop has four corners, \(2(n-1)\) crossings and \(2n+2\)
straight segments.
Put a transport block on every segment with
\[
 \lambda=\epsilon/100,\qquad\mu=\lambda/200,
\]
alternating its forward and reverse placements along each loop.
At the four corners use the signs \(+,-,+,-\) and
\(\sigma=\epsilon/1000\).

Distinct nonincident axial segments have separation at least \(\epsilon\).
The end neighborhoods have radius at most \(\epsilon/100\);
transport widths are at most \(\epsilon/20000\).
Hence the block interiors are disjoint except for their intended
interfaces. Fill the angular gaps, place the construction inside
\([-1,3]^2\), and triangulate the complementary polygons with their
existing vertices. Two visible nonintersecting bridges to the left
outer corners split the exterior annulus into ordinary polygons.
Replace every filler triangle by the fixed seven-point patch.
All templates are embedded in their own coarse cells; the simple outer
square and positive face orientations prove that the final object is a
conforming disk.

Let \(G,M,V_c\) be the numbers of 17-vertex quadrilateral blocks,
triangular filler patches and coarse vertices. Counting ports, segment
corners, corner-block vertices and the outer square gives
\[
 G=4n^2+8n,\qquad V_c=9n^2+23n+4.
\]
The coarse disk has four boundary edges, so
\(2G+M=2V_c-6\), and therefore
\[
 M=10n^2+30n+2.
\]
Each quadrilateral adds thirteen internal vertices; each triangular
patch adds four. Thus
\begin{equation}\label{eq:cross-count}
 V=V_c+13G+4M=101n^2+247n+12,\quad
 N=28G+9M=2V-6,\quad E_I=3V-11.
\end{equation}
This proves the geometric part of \cref{thm:main-unbounded}.

\begin{figure}[htbp]
\centering\includegraphics[width=\linewidth]{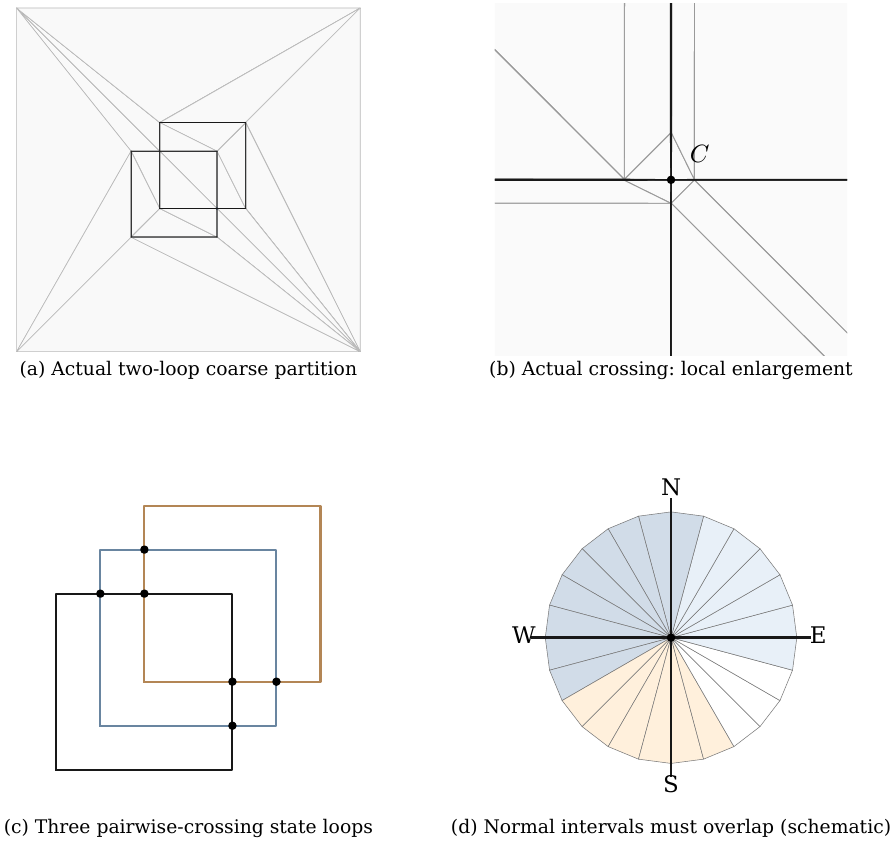}
\caption{Pairwise-crossing construction. Panel (a) is the actual coarse
two-loop partition; every narrow quadrilateral is replaced by its exact
17-vertex block, and every triangle by the seven-point patch.
Panel (b) enlarges an actual coarse crossing.
Panel (c) shows three actual coarse state loops.
Panel (d) illustrates the normal-star interval argument; it is a
topological schematic, not a drawing of the actual fine star.}
\label{fig:crossing}
\end{figure}

\subsection{The complete data image and the full space dimension}
Each loop has \(2n+6\) scalar states and the same number of transport
or turn equations
\[
 h_{j+1}=g_jh_j+\ell_j(f).
\]
The segment gains occur in equal numbers as \(r,r^{-1}\), and the
four corner gains are \(q,q^{-1},q,q^{-1}\).
Their product is exactly one.
Let \(t_0=1\), \(t_{j+1}=g_jt_j\), and write \(h_j=t_jz_j\).
The loop becomes
\[
 z_{j+1}-z_j=\ell_j(f)/t_{j+1}.
\]
It has exactly one compatibility condition,
\begin{equation}\label{eq:loop-relation}
 \Lambda_i(f)=\sum_{j\ {\rm on}\ Q_i}\ell_j(f)/t_{j+1}=0.
\end{equation}
Each local right side has a private internal center with nonzero
coefficient. Therefore the \(n\) conditions are independent.

They are also sufficient on the entire spline space.
Solve the scalar loops, recover the full gradient at each port,
solve each shared midpoint normal variable, realize all quadrilateral
blocks with their full boundary data, and finally fill all triangular
patches with all their given vertex values.
Every complete common-edge trace matches. Consequently
\begin{equation}\label{eq:cross-image}
 \im E_{\mathcal D_n}=\bigcap_{i=0}^{n-1}\ker\Lambda_i,\qquad
 \dd(\mathcal D_n)=n.
\end{equation}

For later use we also need the full space dimension, not just its
vertex image. Its complete coordinates are all \(V\) vertex values,
all \(2V_c\) coarse-vertex gradients, one midpoint derivative on
each of the \(E_c\) coarse edges, eleven zero-data zero-trace kernel
coordinates on each quadrilateral block, and two on each triangular
patch. Impose the one mixed condition on each quadrilateral.
These \(G\) conditions are independent by the private-center coefficients.
Thus
\[
 \dim S_3^1=V+2V_c+E_c+11G+2M-G.
\]
Using \(E_c=V_c+G+M-1\) and \(2G+M=2V_c-6\) reduces this to
\begin{equation}\label{eq:cross-dimension}
 \dim S_3^1=2V+5.
\end{equation}
Equations \eqref{eq:cross-count} and \eqref{eq:cross-dimension} imply
that the native complete \(H\) has full row rank \(3V-11\).
In particular the balanced pressure representing a given vertex
relation is unique. This fact is required for the geometric
indecomposability proof.

\subsection{Why two bad normal submeshes must share a triangle}
Consider the common edge of a mirrored transport pair.
Before eliminating its midpoint derivative the mixed equations contain
\(+\mu\eta n_e\) and \(-\mu\eta n_e\).
A nonzero multiplier of the transport equation therefore has a nonzero
coefficient on the native fine-edge matching row across that edge.
The internal equations of the two blocks do not include that row,
and other segments do not include that edge, so it cannot be canceled
by their internal row combinations.
Every left-kernel coefficient \(1/t_j\) of a loop is nonzero.
Hence a global relation
\[
 R=\sum_i c_i\Lambda_i,\qquad c_i\ne0,
\]
has a nonzero coefficient on \emph{every} transport-edge matching row
of loop \(i\).

Let \(D\) be a connected normal bad submesh, allowing holes.
Extend one of its nonzero vertex relations and its native edge
pressure by zero to the whole mesh.
The native submesh rows are exactly the global rows on its own
interior edges: both adjacent original triangles are present.
The resulting global pressure represents the extended relation.
Uniqueness from \eqref{eq:cross-dimension} forces it to be the
loop pressure just described.
Thus for at least one active loop, every transport edge is interior
to \(D\); both actual fine triangles incident to each such edge belong
to \(D\).

Now take two bad submeshes. If they have the same active loop, they
already share these triangles. Otherwise their active loops cross.
At a crossing the horizontal and vertical transport rays alternate
around the shared vertex. A normal submesh star is a single cyclic
interval or the whole circle, even on a hole boundary.
An interval containing the triangles on both sides of both horizontal
rays must pass through the triangles adjacent to at least one vertical
ray. A vertical-active submesh contains those triangles on both
vertical rays. Their fine triangle sets intersect.
This proves strong indecomposability and completes
\cref{thm:main-unbounded}.

\subsection{Further properties of the loop family}
The first four recorded sizes are
\[
\begin{array}{c|rrrr}
 n&1&2&3&9\\\hline
 V&360&910&1662&10416\\
 \dim S_3^1&725&1825&3329&20837 .
\end{array}
\]
The saved finite certificates check the corresponding actual coarse
partitions, fine topology and state cycles. Their role is separate
from the uniform construction and proof for every \(n\).

Despite its arbitrarily large full-vertex defect, this entire family
interpolates arbitrary data on every collinear subset of its original
vertices, with no bound on the subset size.
Indeed every nonzero coefficient \(c_i\) loads the four private
corner-block centers of loop \(i\), and no other loop loads those
same vertices. If the reference private point is \(O=(x_O,y_O)\),
put \(a=1-x_O\), \(b=\vartheta y_O\); strict interiority gives
\(a>|b|\).
Its four transformed copies form a rectangle with side lengths
\[
 1+2\sigma(a-b),\qquad1+2\sigma(a+b),
\]
both positive. No line contains all four.
Thus no nonzero global relation can have collinear support.
Restriction of evaluation to a subset is onto exactly when no
relation is supported on it, proving the assertion.

The exact gains also describe a genuine fixed-boundary \(n\)-parameter
geometric slice near the constructed mesh.
On one forward segment of each loop replace \(\lambda_0=\epsilon/100\)
by \(\lambda_i\), retaining its endpoints and putting
\(\mu_i=\lambda_i/200\); update neighboring fillers affinely.
Strict inequalities preserve the same outer square and fine face list.
The loop gain becomes \(r(\lambda_i)/r(\lambda_0)\).
Writing \(A=A_0+A_1\lambda\), \(B=B_0+B_1\lambda\), one has
\[
 r'(\lambda)=\frac{A_0B_1-A_1B_0}{B(\lambda)^2}\ne0
\]
on the certified reference interval. Therefore in this slice
\[
 \dd=\#\{i:\lambda_i=\lambda_0\},\qquad
 \Sigma_{\rm slice}=\bigcup_i\{\lambda_i=\lambda_0\}.
\]
Moving every selected parameter off its reference value removes
all constraints. This describes the stated slice, not the complete
bad set in the much larger internal-coordinate space.
 \section{What the examples imply about interpolation geometry}
\label{sec:properties}
\subsection{Relations, subsets and geometric support}
For any vertex subset \(Y\), the evaluation \(E_Y\) is onto if and
only if no nonzero full relation is supported on \(Y\).
Indeed a functional on \(\R^Y\) annihilates \(E_Y\) exactly when its
extension by zero annihilates the full evaluation.
Thus in a defect-one mesh with relation support \(C\),
\[
 E_Y\text{ is onto}\quad\Longleftrightarrow\quad C\not\subset Y.
\]
For the four examples this completely classifies partial vertex data.
All subsets of size at most sixteen are independent in the 17- and
41-vertex cases; all subsets of size at most twelve are independent
in the 25- and 31-vertex cases. Removing an evaluation is distinct
from deleting a mesh vertex or its incident triangles.

More generally, if \(L\) has independent relation rows, orthogonal
projection gives
\[
 \min_s\norm{E_\Delta s-f}_2=
 \norm{L^{\mathsf T}(LL^{\mathsf T})^{-1}Lf}_2,\qquad
 f_{\rm corr}=f-L^{\mathsf T}(LL^{\mathsf T})^{-1}Lf.
\]
The corrected data are unique; the interpolating spline need not be.
At least \(\dd\) vertex evaluations must be removed to obtain an
independent family, and exactly \(\dd\) suffice by choosing an
evaluation basis. For the chains one can remove the core centers.

The support of a relation need not identify all geometry producing it.
The 25- and 31-vertex meshes have the same thirteen core points and
the same relation. Their bare twelve-triangle core,
\((O,A_i,U_i),(O,U_i,A_{i+1})\), has evaluation rank thirteen,
not twelve. Its dimension is 39 and its complete \(H,H_0\) ranks
are both twelve. The outer completion restricts the allowed full
interface traces, thereby creating the core data relation even though
its outer value coefficients are zero.

Every relation annihilates every global cubic:
\begin{equation}\label{eq:cubic-moments}
 \sum_v\lambda_v x_v^a y_v^b=0,\qquad a,b\ge0,\quad a+b\le3.
\end{equation}
This is a necessary consistency check on a computed loading, not a
sufficient spline obstruction criterion.
Any four distinct points admit cubic polynomial interpolation.
Any five distinct points that are not all collinear do as well:
for each selected point choose a line through two of the other four
which avoids it, and two further lines covering the remaining two;
their product is a cubic cardinal polynomial.
If every pair-line of the other four passes through the selected point,
all five points are collinear, the excluded case.
Thus a possible five-point obstruction must be collinear and, along
line parameters \(t_i\), is necessarily a multiple of
\[
 \sum_{i=0}^4
 \frac{f_i}{\prod_{j\ne i}(t_i-t_j)}=0.
\]
Whether such a five-collinear-original-vertex obstruction can occur
on an arbitrary triangulation remains open in the present account.
The loop family in Section~\ref{sec:crossing} has no obstruction on
any collinear subset, however large.

\subsection{Boundary cardinal functions and direct positive criteria}
\begin{lemma}[Low-degree boundary vertex]\label{lem:boundary-cardinal}
If a boundary vertex \(b\) has complete mesh degree at most four,
there is a full \(C^1\) cubic spline \(\psi_b\), supported on its star,
with \(\psi_b(v)=\delta_{bv}\) at every mesh vertex.
\end{lemma}
\begin{proof}
Write the link of \(b\) as \(a_0,\ldots,a_m\), \(m\le3\).
On \(T=(b,a_j,a_{j+1})\), with barycentric coordinate \(\lambda_b\),
set
\[
 \psi_{b,T}(x)=\lambda_b(x)^2
       [3-2\lambda_b(x)+g\cdot(x-p_b)],
\]
and set the function to zero off the star.
Choose one vector \(g\) such that
\[
 g\cdot(p_{a_j}-p_b)=-3,\qquad1\le j<m.
\]
There are at most two equations; if there are two, their directions
bound an actual nondegenerate triangle and are independent.
Along an internal radial edge the two pieces have common value
\(\lambda_b^3\) and common gradient \(\lambda_b^2g\).
Along every link edge both the value and gradient vanish.
The zero extension is consequently \(C^1\).
\end{proof}

Let \(L\) be all these low-degree boundary vertices and
\(A=\mathcal V\setminus L\). If the \(|A|\times10\) cubic polynomial
evaluation matrix on \(A\) has full row rank, then
\[
 s=p+\sum_{b\in L}(f_b-p(p_b))\psi_b
\]
interpolates all vertices, where \(p\in\PP_3\) first matches the values
on \(A\). The construction is analytic under a fixed nonzero
polynomial evaluation minor and permits spline-dimension jumps.
It also explains why a loaded boundary vertex in a bad submesh must
have complete degree at least five, the fact used in \cref{thm:69}.

\subsection{Minimal quadratic dimension still does not imply cubic interpolation}
There are exact bad geometries arbitrarily near the rational
25-vertex example for which both the quadratic and cubic spline
spaces have their lower-bound dimensions:
\[
 \dim S_2^1=9=V_B+3,\qquad
 \dim S_3^1=57,\qquad \rank E_{S_3^1}=24.
\]
In fact the nearby unique relation has all twenty-five coefficients
nonzero, so every proper vertex subset is interpolatory.

For this statement use the true analytic secondary-branch curve
obtained by retaining the \(Q_6,J\) symmetry, fixing \(\eta_i=22\),
and setting the odd radial seam weight to \(t\).
Eight representative balance equations have an invertible
\(8\times8\) Jacobian and define the other eight state variables.
This is an actual bad curve with the same fixed boundary, not its
first-order tangent line.
The complete quadratic cycle matrix \(J_2(t)\) has size \(57\times60\)
and base rank 51. If \(K,N\) are its exact right and left null bases,
the computed first normal matrix
\[
 N^{\mathsf T}J_2'(0)K
\]
has a nonzero \(6\times6\) rational minor.
After eliminating an invertible \(51\times51\) block, the corresponding
\(57\times57\) minor has nonzero leading term of order \(t^6\).
Thus \(J_2(t)\) has full row rank at every sufficiently small nonzero
\(t\), giving the quadratic dimension nine.
The full cubic \(H\) stays full row rank, while the exact balanced
loading keeps \(H_0\) at rank 59.
The previously zero outer loadings have derivatives
\[
 R'_{W_i}(0)=972658125/409876,\quad
 R'_{P_{2j}}(0)=6246383/922221,\quad
 R'_{P_{2j+1}}(0)=-5145,
\]
all nonzero. The thirteen old nonzero coefficients persist, proving
the twenty-five-element circuit.
The quadratic-cycle and physical quadratic-Hermite certificates in
the evidence archive give separate first-order rank calculations.
This example excludes a dimension-only bridge from quadratic
minimality to full cubic vertex interpolation.

\subsection{Combinatorics and order type do not decide the answer}
In the 17-vertex bad example all eighty endpoint coefficients
\(a_{ij}\) are nonzero. There are no singular interior vertices.
Thus neither those nonvanishing tests nor the classical dimension
equality imply interpolation.

There is a sharper exact comparison. Move only
\(y(U_0)\) by \(10^{-5}\), retaining the other exact algebraic
coordinates and the same boundary and faces.
The certified \(40\times40\) minor of the new \(H_0\) is nonzero, so
the new geometry is interpolatory.
All 680 labelled three-point determinant signs are unchanged:
ten are zero and 670 are nonzero in both meshes.
The ten zeros are exactly the triples of the five axis points
\(P_2,A_2,O,A_0,P_0\).
All eighty endpoint coefficient signs also stay unchanged.
Hence the full labelled order type, the abstract triangulation,
and these endpoint zero and sign patterns still do not determine
vertex interpolation. The rational interval certificate records
both the sign separations and the good-point minor.

Ordinary projective transformations do preserve the rank of vertex
evaluation. For a projective map without poles on the domain,
homogeneous degree-three pullback as in \eqref{eq:weighted-pullback}
identifies the full spline spaces, and vertex values are multiplied
by nonzero diagonal factors. The piecewise map used for
\(\Sigma_{17}\) required the additional seam calculation; this
ordinary invariance would not justify it by itself.

\subsection{Generic repair: exact statements and remaining scope}
For any fixed face list and fixed boundary, let \(\theta\) be the
internal coordinates and \(D\) its nonempty legal open realization
domain. The complete matrices have rational entries.
Let their ranks over \(\R(\theta)\) be \(\rho,\rho_0\).
Away from a proper algebraic set their ranks attain these values
simultaneously, in every component of \(D\). Thus
\[
 \delta_{\rm gen}=\rho-\rho_0.
\]
If \(\rho=\rho_0\), good configurations are dense in every legal
neighborhood. If \(\rho>\rho_0\), a legal open dense set is bad.
This is a finite algebraic dichotomy, not a proof that the first
case holds for every possible triangulation and boundary.
The comparison concerns both matrices; a generic dimension theorem
alone controls only \(H\).

The defect need not equal the least number of vertices that must move.
For the chain \(\Delta_m\), every unchanged old core retains its own
obstruction; the core vertex sets are disjoint.
Thus at least \(m\) internal vertices must move to repair all constraints.
Moving one \(U_i\) in each core suffices by the isolated-core result;
the port liftings remain available since those moves do not alter
the flat supporting triangles of the patches.
The least repair count in that family is exactly \(m\).
By contrast the one shared interior vertex of the 69-vertex example
repairs two constraints at once.

There is also a bounded-degree single-node certificate.
Multiply each complete edge equation by \(3[ijk][jil]\), fixing
all labels and orientations, to obtain a polynomial matrix.
When only an internal vertex \(v\) moves, its rows have degree at
most two at an edge endpoint and at most one when \(v\) is the
opposite vertex. If \(d_v\) is its number of incident interior
edges and \(o_v\) its number of appearances as an opposite vertex,
every minor has total degree at most
\[
 D_v=2d_v+o_v\le3\deg(v).
\]
Assume the original \(H\) has full row rank \(m\).
On the triangular grid
\[
 p_v+\tau(a,b),\qquad a,b\ge0,\quad a+b\le D_v,\quad\tau\ne0,
\]
one full-row-rank value of the polynomial \(H_0\) proves arbitrary
small repair by moving \(v\) alone. If all grid values have rank
less than \(m\), every \(m\)-minor vanishes identically in that
vertex plane, so no sufficiently small move of that vertex repairs
the defect. The grid is unisolvent for bivariate degree \(D_v\);
algebraic grid points need not themselves be legal meshes.
This is an exact-rank certificate, not a floating-point search rule.

\begin{theorem}[Fixed-boundary generic interpolation on interior forests]
\label{thm:forest-generic}
If the interior graph is a forest, there is a nonzero polynomial
\[
 P_T(\theta)=\prod_{(a,v,c)\in\mathscr L}[p_a,p_v,p_c],
 \qquad |\mathscr L|\le V_I,\quad\deg P_T\le2V_I,
\]
such that every legal configuration with \(P_T\ne0\) has full-row-rank
\(H,H_0\). In particular every bad configuration in this class admits
arbitrary small joint internal perturbations restoring interpolation.
\end{theorem}
\begin{proof}
Use a topological reduction sequence of the entire current disk.
Remove either a boundary ear, a degree-three boundary vertex whose
middle link vertex is an interior leaf, or an interior degree-three
star adjacent to a true boundary edge.
The inverse ear and noncollinear two-triangle additions match the
entire old interface and give arbitrary new vertex values.
Only the two-triangle step needs
\([p_a,p_v,p_c]\ne0\); collect these triples in \(\mathscr L\).

For the third inverse step, in the macrotriangle
\(u=(0,0),a=(1,0),b=(0,1)\), let the new center be
\(v=(s,t)\), \(s,t>0,s+t<1\).
There is a three-piece cubic with zero old corner jets, center
value one, center gradient \((3/(2s),3/(2t))\), and central BB
coefficients
\[
 (q_{vua},q_{vab},q_{vbu})=(0,1/(2st),0).
\]
Its value and full gradient vanish along \(ua,ub\); the third
macroedge \(ab\) is a true boundary edge. Zero extension is therefore
valid on the entire current mesh. This cardinal function realizes
the new center value without altering the old data.
The three-piece space has dimension twelve and the old cubic
space ten, so this inverse step adds at most two dimensions.
The formal dimension increases by exactly two; dimension lower
bounds then show it adds exactly two when the coarse mesh has
the lower-bound dimension.

To find the reduction, cut along all current boundary-to-boundary
interior chords and choose a leaf region of the resulting disk
region tree. A region without interior vertices is a boundary ear.
A region with a single interior vertex permits the boundary
two-triangle step or the three-star step.
Otherwise its interior tree has at least two leaves; at most one
touches the region's unique parent chord.
At another leaf, either two boundary neighbors give the
three-star step, or a longer boundary-neighbor chain contains a
degree-three boundary vertex for the two-triangle step.
Thus the sequence reaches a single triangle.

Reverse the sequence, preserving every complete interface.
It constructs a full vertex right inverse and the lower-bound
dimension whenever all recorded factors are nonzero.
Every factor contains a distinct originally movable interior
vertex as its middle node, so it is a nonzero polynomial for any
fixed legal boundary. Its degree is at most two, and there are
at most \(V_I\) such factors.
Their product cannot vanish on an open legal coordinate set.
\end{proof}
This theorem is consistent with the 31- and 41-vertex counterexamples:
their actual geometry lies on the exceptional polynomial zero set.
It asserts generic repair, not interpolation at every forest geometry.
Nor does it establish a universal one-node repair theorem for
defect-one meshes with arbitrary interior topology.

\subsection{Further questions}
The four local models and the loop construction leave several distinct
questions. The global minimum vertex count is not determined here;
the 25 and 31 minima proved above are confined to a specified
completion family. A simple geometric characterization outside that
family remains to be found.
For 25 vertices the exact analytic function in
\eqref{eq:25-normal-form}, rather than its cubic leading term alone,
governs the finer error-bound problem.
For arbitrary fixed-boundary topology, the equality
\(\rho=\rho_0\) and universal defect-one single-node repair are not
settled by the examples or finite witness methods.
Finally, the strongly indecomposable construction fixes its outer
square but makes its transport channels increasingly thin;
it supplies no minimum-angle bound uniform in \(n\).
The chain family supplies that angle control while keeping the
obstructions geometrically separable.
 
 \end{document}